\documentclass[12pt]{amsart}
\usepackage{amssymb}
\usepackage[all]{xy}
\usepackage{tikz}
\usetikzlibrary {positioning}
\usepackage{verbatim}
\usepackage[draft]{todonotes}
\usepackage{amsmath}
\usepackage{amsthm}
\usepackage{mathrsfs}
\usetikzlibrary{arrows}
\usepackage{blindtext}
\newcounter{diagram}
\numberwithin{diagram}{section}

\newenvironment{diagram}
  {\stepcounter{diagram}\par\smallskip\noindent\begin{minipage}{\linewidth}\centering}
  {\par Diagram~\thediagram\end{minipage}\par\smallskip}
\input xy
\xyoption{all}

\newtheorem{lemma}{Lemma}[section]

\newtheorem{theorem}[lemma]{Theorem}
\newtheorem{proposition}[lemma]{Proposition}

\theoremstyle{definition}
\newtheorem{remark}[lemma]{Remark}
\newtheorem{definition}[lemma]{Definition}

\DeclareMathOperator{\modd}{mod}
\DeclareMathOperator{\Endd}{End}

\DeclareMathOperator{\Hom}{Hom}

\DeclareMathOperator{\Ext}{Ext}

\DeclareMathOperator{\Ker}{Ker}
\DeclareMathOperator{\Imm}{Im}

\newtheorem{question}[lemma]{Question}
\newtheorem*{theorem a*}{Theorem A}
\newtheorem*{theorem b*}{Theorem B}

\begin{document}

\title{representation of $n$-abelian categories in abelian categories}

\author{Ramin Ebrahimi}
\address{Department of Pure Mathematics\\
Faculty of Mathematics and Statistics\\
University of Isfahan\\
Isfahan 81746-73441, Iran}
\email{ramin69@sci.ui.ac.ir}

\author{Alireza Nasr-Isfahani}
\address{Department of Pure Mathematics\\
Faculty of Mathematics and Statistics\\
University of Isfahan\\
Isfahan 81746-73441, Iran\\ and School of Mathematics, Institute for Research in Fundamental Sciences (IPM), P.O. Box: 19395-5746, Tehran, Iran}
\email{nasr$_{-}$a@sci.ui.ac.ir / nasr@ipm.ir}

\subjclass[2010]{{18E10}, {18E20}, {18E99}}

\keywords{abelian category, $n$-abelian category, embedding theorem, homological algebra, higher homological algebra}

\begin{abstract}
We prove a higher-dimensional version of the Freyd-Mitchell embedding theorem for $n$-abelian categories. More precisely, for a positive integer $n$ and a small $n$-abelian category $\mathcal{M}$, we show that $\mathcal{M}$ is equivalent to a full subcategory of an abelian category $\mathcal{L}_2(\mathcal{M},\mathcal{G})$, where $\mathcal{L}_2(\mathcal{M},\mathcal{G})$ is the category of absolutely pure group valued functors over $\mathcal{M}$. We also show that $n$-kernels and $n$-cokernels in $\mathcal{M}$ are precisely exact sequences of $\mathcal{L}_2(\mathcal{M},\mathcal{G})$ with terms in $\mathcal{M}$.
\end{abstract}

\maketitle


\section{Introduction}
For a positive integer $n$, $n$-cluster tilting subcategories of abelian categories were introduced by Iyama \cite{I2, I1}(see also \cite{I3, I4} and \cite{JK}) to develop the higher-dimensional analogs of Auslander-Reiten theory.
In these subcategories kernels and cokernels don't necessarily exist and are replaced with $n$-kernels and $n$-cokernels. Recently, Jasso \cite{J} introduced $n$-abelian categories which are an axiomatization of $n$-cluster tilting subcategories. He proved that any $n$-cluster tilting subcategory of an abelian category is $n$-abelian.

Freyd in \cite{Fr2} asked the
following general question, "Given a category how nicely can it be
represented in an abelian category?" Note that in homological
algebra it is very convenient to have a concrete abelian category,
for that allows one to check the behavior of morphisms on actual
elements of the sets underlying the objects. Freyd in \cite[Chapter 7]{Fr}
proved that for every small abelian category there exists an
exact, covariant embedding into the category of abelian groups
(see also \cite{Lub}). This shows that if a statement concerning
exactness and commutativity of a diagram is true in the category
of abelian groups then it is true in a small abelian category.
If a small abelian category $\mathcal{A}$ has a projective generator
$P,$ then there is an exact fully faithful functor from
$\mathcal{A}$ to the category of modules over
$\Endd_{\mathcal{A}}(P)$ and if $\mathcal{A}$ has a generating set
of projectives, then there is an exact fully faithful functor from
$\mathcal{A}$ to the category of modules over a ring with several
objects \cite{Fr, Mi}. Mitchell in \cite[Theorem 4.4]{Mi2} proved that, "every
small abelian category admits a full, exact and covariant
embedding into a category of $R$-modules for some ring $R$". The
fullness of the embedding functor implies that the statement
concerning the existence of morphisms in a diagram in a small
abelian category is true providing that it is true in the category
of modules. Let $\mathcal{A}$ be a small abelian category, $\mathcal{G}$ be the category of all abelian groups and
$(\mathcal{A},\mathcal{G})$ be the category of additive group
valued functors. A left exact sequence is an exact sequence of the form $0\rightarrow A_1\rightarrow A_2\rightarrow A_3$. A functor $F\in (\mathcal{A},\mathcal{G})$ which carries left exact sequences to left exact sequences is called a left exact functor. Let $\mathcal{L}(\mathcal{A},\mathcal{G})$ be the
full subcategory of $(\mathcal{A},\mathcal{G})$ consisting of all
left exact functors. Gabriel proved that
$\mathcal{L}(\mathcal{A},\mathcal{G})$ is an abelian category with
an injective cogenerator \cite{G} and Mitchell proved that the Yoneda functor
$\mathcal{A}\rightarrow \mathcal{L}(\mathcal{A},\mathcal{G})$ is
exact and fully faithful \cite{Mi2}.

It is natural to extend the Freyd-Mitchell embedding theorem to $n$-abelian categories. Jasso in \cite[Theorem 3.20]{J} proved that $n$-abelian categories satisfying certain mild assumptions can be realized as $n$-cluster tilting subcategories of
abelian categories. Let $\mathcal{M}$ be a small $n$-abelian category with a generating set of projectives, $\mathcal{P}$ be the category of projective objects in $\mathcal{M}$ and $F:\mathcal{M}\longrightarrow\modd\mathcal{P}$ be the functor defined by $F(X)=\mathcal{M}(-,X)|_\mathcal{P}$. Jasso in \cite[Theorem 3.20]{J} proved that if there exists an exact duality $D:\modd\mathcal{P}\longrightarrow\modd\mathcal{P}^{op}$, then $F$ is fully faithful and the essential image of $F$ is an $n$-cluster tilting subcategory of the abelian category $\modd\mathcal{P}$. Kvamme in \cite{K} proved that the existence of the exact duality is unnecessary.

In this paper we show that any small $n$-abelian category can be embedded in an abelian category in such a way that the embedding functor is $n$-exact and reflects $n$-exact sequences. Since we want to prove the embedding theorem for a small $n$-abelian category, the Jasso's method in \cite[Theorem 3.20]{J} don't work in this situation and our method in this paper is different from Jasso's method in \cite[Theorem 3.20]{J}.

For a small $n$-abelian category $\mathcal{M}$ we consider the category of additive group valued functors $(\mathcal{M},\mathcal{G})$ and the representation functor $H:\mathcal{M}\rightarrow (\mathcal{M},\mathcal{G})$. A functor $F\in (\mathcal{M},\mathcal{G})$ is called a mono functor if it preserves monomorphisms. We denote by $\mathbb{M}(\mathcal{M},\mathcal{G})$, the full subcategory of $(\mathcal{M},\mathcal{G})$ consist of all mono functors. A subfunctor $F^{\prime}\subseteq F$ in $\mathbb{M}(\mathcal{M},\mathcal{G})$ is called a pure subfunctor if the quotient functor $\dfrac{F}{F^{\prime}}\in \mathbb{M}(\mathcal{M},\mathcal{G})$. A mono functor is called absolutely pure if whenever it appears as a subfunctor of a mono functor it is a pure subfunctor. We denote by $\mathcal{L}_2(\mathcal{M},\mathcal{G})\subseteq \mathbb{M}(\mathcal{M},\mathcal{G})$ the full subcategory of absolutely pure functors. By \cite[Theorem 7.31]{Fr}, $\mathcal{L}_2(\mathcal{M},\mathcal{G})$ is an abelian category and every object of $\mathcal{L}_2(\mathcal{M},\mathcal{G})$ has an injective envelope.

In the following theorem we give a characterization of absolutely pure functors. We refer the reader to Section 2 for the definition of left $n$-exact sequence. The following theorem is the higher analogue, on $n$-abelian categories, of \cite[Theorem 7.27]{Fr}, which is on abelian categories.
\newtheorem{thm}{Theorem}
\begin{theorem a*}$($See Theorem \ref{theoremc}$)$
A mono functor $M\in (\mathcal{M},\mathcal{G})$ is absolutely pure if and only if whenever $X^0\rightarrow X^1 \rightarrow \ldots \rightarrow X^n\rightarrow X^{n+1}$ is a left $n$-exact sequence in $\mathcal{M}$, then $0\rightarrow M(X^0)\rightarrow M(X^1)\rightarrow M(X^2)$ is an exact sequence of abelian groups.
\end{theorem a*}

The representation functor $H:\mathcal{M}\rightarrow (\mathcal{M},\mathcal{G})$ is only left $n$-exact but not $n$-exact (See Definition \ref{def1}). To fix this, we use the above theorem and replace $(\mathcal{M},\mathcal{G})$ with a subcategory $\mathcal{L}_2(\mathcal{M},\mathcal{G})$ of $(\mathcal{M},\mathcal{G})$. The functor $H$ sends the object $X\in\mathcal{M}$ to the functor $\mathcal{M}(X,-)$, which is a left $n$-exact functor. Then by the above theorem, $\mathcal{L}_2(\mathcal{M},\mathcal{G})$ contains the essential image of $H$.

Using these results, we prove the following theorem, which is the main result of this paper.

\begin{theorem b*}$($See Theorems \ref{theoremmain3} and \ref{theoremmain2}$)$
The full embedding functor $H:\mathcal{M}\rightarrow \mathcal{L}_2(\mathcal{M},\mathcal{G})$ is $n$-exact and reflects $n$-exact sequences.
\end{theorem b*}

Then we can describe each small $n$-abelian category $\mathcal{M}$ as a full subcategory of an abelian category $\mathcal{L}_2(\mathcal{M},\mathcal{G})$ such that left $n$-exact sequences, right $n$-exact sequences and $n$-exact sequences are precisely exact sequences of the abelian category $\mathcal{L}_2(\mathcal{M},\mathcal{G})$ with terms in $\mathcal{M}$. In $n$-abelian categories we have $n$-exact sequences and the above theorem shows that any statement concerning ($n$-)exactness of a finite diagram, commutativity of a finite diagram and the existence of morphisms in a finite diagram in a small $n$-abelian category is true providing that the corresponding statement is true in the abelian categories. This gives a useful connection between (higher-)homological properties of $n$-abelian categories and homological properties of abelian categories.

The paper is organized as follows. In Section 2 we recall the definitions of $n$-abelian categories and $n$-cluster tilting subcategories and recall some results that we need in the paper. In Section 3 we consider the category of functors from a small $n$-abelian category to the category of abelian groups, then we introduce the subcategory of mono functors and the subcategory of absolutely pure functors as in the classic case. We show that the subcategory of absolutely pure functors is an abelian category. In Section 4 we characterize absolutely pure functors and show that the representation functor from a small $n$-abelian category to the category of absolutely pure functors is a full embedding which preserves and reflects $n$-exactness. Finally in Section 5 we give some applications of our main result and prove some homological results for a small $n$-abelian category.

\subsection{Notation}
Throughout this paper, unless otherwise stated, $n$ always denotes a fixed positive integer, $\mathcal{M}$ is a fixed small $n$-abelian category and $\mathcal{G}$ is the abelian category of all abelian groups.


\section{$n$-abelian categories}
In this section we recall the definition of $n$-abelian category and recall some results that we need in the rest of the paper. For further information and motivation of definitions the readers are referred to \cite{I1, I2, J}.

\subsection{$n$-abelian categories}
Let $\mathcal{M}$ be an additive category and $f:A\rightarrow B$ a morphism in $\mathcal{M}$. A weak cokernel of $f$ is a morphism $g:B\rightarrow C$ such that for all $C^{\prime} \in \mathcal{M}$  the sequence of abelian groups
\begin{equation}
\mathcal{M}(C,C')\overset{(g,C')}{\longrightarrow} \mathcal{M}(B,C')\overset{(f,C')}{\longrightarrow} \mathcal{M}(A,C') \notag
\end{equation}
is exact. The concept of weak kernel is defined dually.

Let $d^0:X^0 \rightarrow X^1$ be a morphism in $\mathcal{M}$. An $n$-cokernel of $d^0$ is a sequence
\begin{equation}
(d^1, \ldots, d^n): X^1 \overset{d^1}{\rightarrow} X^2 \overset{d^2}{\rightarrow}\cdots \overset{d^{n-1}}{\rightarrow} X^n \overset{d^n}{\rightarrow} X^{n+1} \notag
\end{equation}
of objects and morphisms in $\mathcal{M}$ such that for all $Y\in \mathcal{M}$
the induced sequence of abelian groups
\begin{align}
0 \rightarrow \mathcal{M}(X^{n+1},Y) \rightarrow \mathcal{M}(X^n,Y) \rightarrow\cdots\rightarrow \mathcal{M}(X^1,Y) \rightarrow \mathcal{M}(X^0,Y) \notag
\end{align}
is exact. Equivalently, the sequence $(d^1, \ldots, d^n)$ is an $n$-cokernel of $d^0$ if for all $1\leq k\leq n-1$
the morphism $d^k$ is a weak cokernel of $d^{k-1}$, and $d^n$ is moreover a cokernel of $d^{n-1}$ \cite[Definition 2.2]{J}. The concept of $n$-kernel of a morphism is defined dually.
\begin{definition}(\cite[Definition 2.4]{L})
Let $\mathcal{M}$ be an additive category. A right $n$-exact sequence in $\mathcal{M}$ is a complex
\begin{equation}
X^0 \overset{d^0}{\rightarrow} X^1 \overset{d^1}{\rightarrow} \cdots \overset{d^{n-1}}{\rightarrow} X^n \overset{d^n}{\rightarrow} X^{n+1} \notag
\end{equation}
such that $(d^1, \ldots, d^{n})$ is an $n$-cokernel of $d^0$. The concept of left $n$-exact sequence is defined dually. An $n$-exact sequence is a sequence which is both a right $n$-exact sequence and a left $n$-exact sequence.
\end{definition}

Let $\mathcal{M}$ be a category and $A$ be an object of $\mathcal{M}$. A morphism $e\in \mathcal{M}(A, A)$ is idempotent if $e^2 = e$. $\mathcal{M}$ is called idempotent complete
if for every idempotent $e\in \mathcal{M}(A, A)$ there exist an object $B$ and morphisms $f\in\mathcal{M}(A, B)$ and
$g\in\mathcal{M}(B, A)$ such that $gf = e$ and $fg = 1_B$.

\begin{definition}$($\cite[Definition 3.1]{J}$)$
An $n$-abelian category is an additive category $\mathcal{M}$ which satisfies the following axioms:
\begin{itemize}
\item[(A0)]
The category $\mathcal{M}$ is idempotent complete.
\item[(A1)]
Every morphism in $\mathcal{M}$ has an $n$-kernel and an $n$-cokernel.
\item[(A2)]
For every monomorphism $d^0:X^0 \rightarrow X^1$ in $\mathcal{M}$ and for every $n$-cokernel $(d^1, \ldots, d^n)$ of $d^0$, the  following sequence is $n$-exact:
\begin{equation}
X^0 \overset{d^0}{\rightarrow} X^1 \overset{d^1}{\rightarrow} \cdots \overset{d^{n-1}}{\rightarrow} X^n \overset{d^n}{\rightarrow} X^{n+1} \notag.
\end{equation}
\item[(A3)]
For every epimorphism $d^n:X^n \rightarrow X^{n+1}$ in $\mathcal{M}$ and for every $n$-kernel $(d^0, \ldots, d^{n-1})$ of $d^n$, the  following sequence is $n$-exact:
\begin{equation}
X^0 \overset{d^0}{\rightarrow} X^1 \overset{d^1}{\rightarrow} \cdots \overset{d^{n-1}}{\rightarrow} X^n \overset{d^n}{\rightarrow} X^{n+1} \notag.
\end{equation}
\end{itemize}
\end{definition}

A subcategory $\mathcal{B}$ of an abelian category $\mathcal{A}$ is called cogenerating if for every object
$X\in \mathcal{A}$ there exists an object $Y\in\mathcal{B}$ and a monomorphism $X\rightarrow Y$. The concept of
generating subcategory is defined dually.

Let $\mathcal{A}$ be an additive category and $\mathcal{B}$ be a full subcategory of $\mathcal{A}$. $\mathcal{B}$ is called
covariantly finite in $\mathcal{A}$ if for every $A\in \mathcal{A}$ there exists an object $B\in\mathcal{B}$ and a morphism
$f : A\rightarrow B$ such that, for all $B'\in\mathcal{B}$, the sequence of abelian groups $\mathcal{A}(B, B')\rightarrow \mathcal{A}(A, B')\rightarrow 0$ is exact. The notion of contravariantly
finite subcategory of $\mathcal{A}$ is defined dually. A functorially
finite subcategory of $\mathcal{A}$ is a subcategory which is both covariantly and contravariantly finite
in $\mathcal{A}$.

\begin{definition}$($\cite[Definition 3.14]{J}$)$
Let $\mathcal{A}$ be an abelian category and $\mathcal{M}$ be a generating-cogenerating full subcategory of $\mathcal{A}$. $\mathcal{M}$ is called an $n$-cluster tilting subcategory of $\mathcal{A}$ if $\mathcal{M}$ is functorially finite in $\mathcal{A}$ and
\begin{align}
\mathcal{M}& = \{ X\in \mathcal{A} \mid \forall i\in \{1, \ldots, n-1 \}, \Ext^i(X,\mathcal{M})=0 \}\notag \\
                  & =\{ X\in \mathcal{A} \mid \forall i\in \{1, \ldots, n-1 \}, \Ext^i(\mathcal{M},X)=0 \}.\notag
\end{align}

Note that $\mathcal{A}$ itself is the unique 1-cluster tilting subcategory of $\mathcal{A}$.

\end{definition}

The following result gives a rich source of $n$-abelian categories.

\begin{theorem}$($\cite[Theorem 3.16]{J}$)$
Let $\mathcal{A}$ be an abelian category and $\mathcal{M}$ be an $n$-cluster tilting subcategory of $\mathcal{A}$. Then $\mathcal{M}$ is an $n$-abelian category.
\end{theorem}

Let $\mathcal{M}$ be an additive category and consider the following morphism of complexes
\begin{center}
\begin{tikzpicture}
\node (X1) at (-4,1) {$X$};
\node (X2) at (-2,1) {$X^0$};
\node (X3) at (0,1) {$X^1$};
\node (X4) at (2,1) {$\ldots$};
\node (X5) at (4,1) {$X^{n-1}$};
\node (X6) at (6,1) {$X^n$};
\node (X7) at (-4,-1) {$Y$};
\node (X8) at (-2,-1) {$Y^0$};
\node (X9) at (0,-1) {$Y^1$};
\node (X10) at (2,-1) {$\ldots$};
\node (X11) at (4,-1) {$Y^{n-1}$};
\node (X12) at (6,-1) {$Y^n$};
\draw [->,thick] (X1) -- (X7) node [midway,left] {$f$};
\draw [->,thick] (X2) -- (X8) node [midway,left] {$f^0$};
\draw [->,thick] (X3) -- (X9) node [midway,left] {$f^1$};
\draw [->,thick] (X5) -- (X11) node [midway,left] {$f^{n-1}$};
\draw [->,thick] (X6) -- (X12) node [midway,left] {$f^n$};
\draw [->,thick] (X2) -- (X3) node [midway,above] {$d_X^0$};
\draw [->,thick] (X3) -- (X4) node [midway,above] {$d_X^1$};
\draw [->,thick] (X4) -- (X5) node [midway,above] {$d_X^{n-2}$};
\draw [->,thick] (X5) -- (X6) node [midway,above] {$d_X^{n-1}$};
\draw [->,thick] (X8) -- (X9) node [midway,above] {$d_Y^0$};
\draw [->,thick] (X9) -- (X10) node [midway,above] {$d_Y^1$};
\draw [->,thick] (X10) -- (X11) node [midway,above] {$d_Y^{n-2}$};
\draw [->,thick] (X11) -- (X12) node [midway,above] {$d_Y^{n-1}$};
\end{tikzpicture}
\end{center}
Recall that the mapping cone $C(f)$ of $f$ is the following complex
\begin{equation}
X^0\overset{d_C^0}{\longrightarrow} X^1\oplus Y^0 \overset{d_C^1}{\longrightarrow} \cdots \overset{d_C^{n-1}}{\longrightarrow} X^n\oplus Y^{n-1}\overset{d_C^n}{\longrightarrow} Y^n \notag,
\end{equation}
where $$d_C^k=\left(\begin{array}{cc}
-d_X^k &0 \\
f^k &d_Y^{k-1}\\
\end{array}\right):X^k\oplus Y^{k-1}\longrightarrow X^{k+1}\oplus Y^k,$$ for each $0\leq k\leq n$.

\begin{definition}$($\cite[Definition 2.11]{J}$)$
Let $\mathcal{M}$ be an additive category, $X: X^0 \overset{d^0}{\rightarrow} X^1 \overset{d^1}{\rightarrow}\cdots \overset{d^{n-2}}{\rightarrow} X^{n-1}\overset{d^{n-1}}{\rightarrow} X^n$ a complex in $\mathcal{M}$ and $f^0:X^0\rightarrow Y^0$ a morphism in $\mathcal{M}$. An $n$-pushout diagram of $X$ along $f^0$ is a morphism of complexes
\begin{center}
\begin{tikzpicture}
\node (X1) at (-4,1) {$X$};
\node (X2) at (-2,1) {$X^0$};
\node (X3) at (0,1) {$X^1$};
\node (X4) at (2,1) {$\ldots$};
\node (X5) at (4,1) {$X^{n-1}$};
\node (X6) at (6,1) {$X^n$};
\node (X7) at (-4,-1) {$Y$};
\node (X8) at (-2,-1) {$Y^0$};
\node (X9) at (0,-1) {$Y^1$};
\node (X10) at (2,-1) {$\ldots$};
\node (X11) at (4,-1) {$Y^{n-1}$};
\node (X12) at (6,-1) {$Y^n$};
\draw [->,thick] (X1) -- (X7) node [midway,left] {$f$};
\draw [->,thick] (X2) -- (X8) node [midway,left] {$f^0$};
\draw [->,thick] (X3) -- (X9) node [midway,left] {};
\draw [->,thick] (X5) -- (X11) node [midway,left] {};
\draw [->,thick] (X6) -- (X12) node [midway,left] {};
\draw [->,thick] (X2) -- (X3) node [midway,above] {};
\draw [->,thick] (X3) -- (X4) node [midway,above] {};
\draw [->,thick] (X4) -- (X5) node [midway,above] {};
\draw [->,thick] (X5) -- (X6) node [midway,above] {};
\draw [->,thick] (X8) -- (X9) node [midway,above] {};
\draw [->,thick] (X9) -- (X10) node [midway,above] {};
\draw [->,thick] (X10) -- (X11) node [midway,above] {};
\draw [->,thick] (X11) -- (X12) node [midway,above] {};
\end{tikzpicture}
\end{center}
such that the mapping cone $C=C(f)$ of $f$ is a right $n$-exact sequence. The concept of $n$-pullback diagram is defined dually.
\end{definition}

\begin{theorem}$($\cite[Theorem 3.8]{J}$)$ \label{theorem1}
Let $\mathcal{M}$ be an additive category which satisfies axioms $(A0)$ and $(A1)$ of $n$-abelian category,
\begin{equation}
X: X^0 \overset{d^0}{\rightarrow} X^1 \overset{d^1}{\rightarrow}\cdots \overset{d^{n-1}}{\rightarrow} X^n  \notag
\end{equation}
a complex in $\mathcal{M}$ and $f^0:X^0\rightarrow Y^0$ a morphism in $\mathcal{M}$. Then the following statements hold:
\begin{itemize}
\item[(i)]
There exists an n-pushout diagram
\begin{center}
\begin{tikzpicture}
\node (X2) at (-2,1) {$X^0$};
\node (X3) at (0,1) {$X^1$};
\node (X4) at (2,1) {$\ldots$};
\node (X5) at (4,1) {$X^{n-1}$};
\node (X6) at (6,1) {$X^n$};
\node (X8) at (-2,-1) {$Y^0$};
\node (X9) at (0,-1) {$Y^1$};
\node (X10) at (2,-1) {$\ldots$};
\node (X11) at (4,-1) {$Y^{n-1}$};
\node (X12) at (6,-1) {$Y^n$};
\draw [->,thick] (X2) -- (X8) node [midway,left] {$f^0$};
\draw [->,thick] (X3) -- (X9) node [midway,left] {};
\draw [->,thick] (X5) -- (X11) node [midway,left] {};
\draw [->,thick] (X6) -- (X12) node [midway,left] {};
\draw [->,thick] (X2) -- (X3) node [midway,above] {$d_X^0$};
\draw [->,thick] (X3) -- (X4) node [midway,above] {};
\draw [->,thick] (X4) -- (X5) node [midway,above] {};
\draw [->,thick] (X5) -- (X6) node [midway,above] {};
\draw [->,thick] (X8) -- (X9) node [midway,above] {$d_Y^0$};
\draw [->,thick] (X9) -- (X10) node [midway,above] {};
\draw [->,thick] (X10) -- (X11) node [midway,above] {};
\draw [->,thick] (X11) -- (X12) node [midway,above] {};
\end{tikzpicture}
\end{center}
\item[(ii)]
Suppose, moreover, that $\mathcal{M}$ is an n-abelian category. If $d_X^0$ is a monomorphism, then $d_Y^0$ is also a monomorphism.
\end{itemize}
\end{theorem}

In the following crucial proposition we give a necessarily and sufficient conditions for when a complex in an $n$-abelian category is a right $n$-exact sequence. One direction of the proposition has been proven in \cite[Proposition 3.13]{J} but we present again part of the argument for convenience of the reader. Note that in the proof we do not need to use the notion of good $n$-pushout diagram (see \cite[Definition-Proposition 2.14]{J} and \cite[Proposition 3.13]{J}).
\begin{proposition}\label{propositionmain}
Let $\mathcal{M}$ be an $n$-abelian category and
\begin{equation}
X: X^0 \overset{f^0}{\rightarrow} X^1 \overset{f^1}{\rightarrow} \cdots \overset{f^{n-1}}{\rightarrow} X^n \overset{f^n}{\rightarrow} X^{n+1} \notag
\end{equation}
a complex in $\mathcal{M}$. Then, for every $k \in \{0, 1, \ldots, n\}$ and every $l \in \{1, 2, \ldots, n\}$ there exist morphisms $g_k^l:Y_k^l\rightarrow Y_k^{l-1}$ (with $Y_k^0=X^k$) and $p_k^{l-1}:Y_k^{l-1}\rightarrow Y_{k+1}^l$ satisfying the following properties:
\begin{itemize}
\item[(i)]
For every $k \in \{0, 1, \ldots, n\}$ the diagram
\begin{center}
\begin{tikzpicture}
\node (X1) at (-4,1) {$Y_k^n$};
\node (X2) at (-2,1) {$Y_k^{n-1}$};
\node (X3) at (0,1) {$\ldots$};
\node (X4) at (2,1) {$Y_k^1$};
\node (X5) at (4,1) {$X^k$};
\node (X6) at (6,1) {$X^{k+1}$};
\node (X7) at (-4,-1) {$0$};
\node (X8) at (-2,-1) {$Y_{k+1}^n$};
\node (X9) at (0,-1) {$\ldots$};
\node (X10) at (2,-1) {$Y_{k+1}^2$};
\node (X11) at (4,-1) {$Y_{k+1}^1$};
\draw [->,thick] (X1) -- (X2) node [midway,above] {$g_k^n$};
\draw [->,thick] (X2) -- (X3) node [midway,above] {$g_k^{n-1}$};
\draw [->,thick] (X3) -- (X4) node [midway,above] {$g_k^2$};
\draw [->,thick] (X4) -- (X5) node [midway,above] {$g_k^1$};
\draw [->,thick] (X5) -- (X6) node [midway,above] {$f^k$};
\draw [->,thick] (X1) -- (X7) node [midway,above] {};
\draw [->,thick] (X2) -- (X8) node [midway,left] {$p_k^{n-1}$};
\draw [->,thick] (X4) -- (X10) node [midway,left] {$p_k^1$};
\draw [->,thick] (X5) -- (X11) node [midway,left] {$p_k^0$};
\draw [->,thick] (X7) -- (X8) node [midway,above] {};
\draw [->,thick] (X8) -- (X9) node [midway,above] {$g_{k+1}^n$};
\draw [->,thick] (X9) -- (X10) node [midway,above] {$g_{k+1}^3$};
\draw [->,thick] (X10) -- (X11) node [midway,above] {$g_{k+1}^2$};
\draw [->,thick] (X11) -- (X6) node [midway,right] {$g_{k+1}^1$};
\end{tikzpicture}
\end{center}
commutes.
\item[(ii)]
The sequence $(g_k^n, \ldots, g_k^1)$ is an $n$-kernel of $f^k$.
\item[(iii)]
The diagram
\begin{diagram}\label{dia3}
\centering
\begin{tikzpicture}
\node (X1) at (-4,1) {$Y_k^n$};
\node (X2) at (-2,1) {$Y_k^{n-1}$};
\node (X3) at (0,1) {$\ldots$};
\node (X4) at (2,1) {$Y_k^1$};
\node (X5) at (4,1) {$X^k$};
\node (X7) at (-4,-1) {$0$};
\node (X8) at (-2,-1) {$Y_{k+1}^n$};
\node (X9) at (0,-1) {$\ldots$};
\node (X10) at (2,-1) {$Y_{k+1}^2$};
\node (X11) at (4,-1) {$Y_{k+1}^1$};
\draw [->,thick] (X1) -- (X2) node [midway,above] {$g_k^n$};
\draw [->,thick] (X2) -- (X3) node [midway,above] {$g_k^{n-1}$};
\draw [->,thick] (X3) -- (X4) node [midway,above] {$g_k^2$};
\draw [->,thick] (X4) -- (X5) node [midway,above] {$g_k^1$};
\draw [->,thick] (X1) -- (X7) node [midway,above] {};
\draw [->,thick] (X2) -- (X8) node [midway,left] {$p_k^{n-1}$};
\draw [->,thick] (X4) -- (X10) node [midway,left] {$p_k^1$};
\draw [->,thick] (X5) -- (X11) node [midway,left] {$p_k^0$};
\draw [->,thick] (X7) -- (X8) node [midway,above] {};
\draw [->,thick] (X8) -- (X9) node [midway,above] {$g_{k+1}^n$};
\draw [->,thick] (X9) -- (X10) node [midway,above] {$g_{k+1}^3$};
\draw [->,thick] (X10) -- (X11) node [midway,above] {$g_{k+1}^2$};
\end{tikzpicture}
\end{diagram}
is an $n$-pullback diagram.
\end{itemize}
Moreover the morphism $[p_k^0, g_{k+1}^2]:X^k\oplus Y_{k+1}^2\rightarrow Y_{k+1}^1$ is an epimorphism for every $k$ if and only if the complex $X$ is a right $n$-exact sequence. In this case we can choose the objects $Y_k^l$, $1\leq l\leq n$ and morphisms $g_k^l$, $1\leq l\leq n$ in such a way that the Diagram 2.1 is both $n$-pullback and $n$-pushout diagram.
\begin{itemize}
\item[(iv)]
If the complex $X$ is a right $n$-exact sequence and $k\neq 0$, then the sequence $(g_k^{k-1}, \ldots, g_k^1, f^k, \ldots, f^n)$ is an $n$-cokernel of the morphism $g_k^k$.
\end{itemize}
\begin{proof}
Let
\begin{equation}
X: X^0\overset{f^0}{\longrightarrow}X^1\overset{f^1}{\longrightarrow}\cdots \overset{f^{n-1}}{\longrightarrow}X^n\overset{f^n}{\longrightarrow}X^{n+1} \notag
\end{equation}
be a right $n$-exact sequence. Since $f^n$ is an epimorphism, there exists an $n$-exact sequence
\begin{equation}
Y_n^n\overset{g_n^n}{\longrightarrow}Y_n^{n-1}\overset{g_n^{n-1}}{\longrightarrow}\cdots \overset{g_n^2}{\longrightarrow}Y_n^1\overset{g_n^1}{\longrightarrow}X^n\overset{f^n}{\longrightarrow}X^{n+1}. \notag
\end{equation}
This implies that the diagram
\begin{center}
\begin{tikzpicture}
\node (X1) at (-4,1) {$Y_n^n$};
\node (X2) at (-2,1) {$Y_n^{n-1}$};
\node (X3) at (0,1) {$\ldots$};
\node (X4) at (2,1) {$Y_n^1$};
\node (X5) at (4,1) {$X^n$};
\node (X6) at (-4,-1) {$0$};
\node (X7) at (-2,-1) {$0$};
\node (X8) at (0,-1) {$\ldots$};
\node (X9) at (2,-1) {$0$};
\node (X10) at (4,-1) {$X^{n+1}$};
\draw [->,thick] (X1) -- (X2) node [midway,above] {$g_n^n$};
\draw [->,thick] (X2) -- (X3) node [midway,above] {$g_n^{n-1}$};
\draw [->,thick] (X3) -- (X4) node [midway,above] {$g_n^2$};
\draw [->,thick] (X4) -- (X5) node [midway,above] {$g_n^1$};
\draw [->,thick] (X6) -- (X7) node [midway,above] {};
\draw [->,thick] (X7) -- (X8) node [midway,left] {};
\draw [->,thick] (X8) -- (X9) node [midway,left] {};
\draw [->,thick] (X9) -- (X10) node [midway,left] {};
\draw [->,thick] (X1) -- (X6) node [midway,above] {};
\draw [->,thick] (X2) -- (X7) node [midway,above] {};
\draw [->,thick] (X3) -- (X8) node [midway,above] {};
\draw [->,thick] (X4) -- (X9) node [midway,above] {};
\draw [->,thick] (X5) -- (X10) node [midway,right] {$f^n$};
\end{tikzpicture}
\end{center}
is both an $n$-pullback diagram and an $n$-pushout diagram. Now by induction assume that for $1\leq k\leq n$ we have the following commutative diagram with the required properties.
\begin{center}
\begin{tikzpicture}
\node (X1) at (-6,4) {$0$};
\node (X2) at (-4,4.) {$Y_{k-1}^n$};
\node (X3) at (-2,4) {$Y_{k-1}^{n-1}$};
\node (X4) at (0,4) {$\ldots$};
\node (X5) at (2,4) {$Y_{k-1}^2$};
\node (X6) at (4,4) {$Y_{k-1}^1$};
\node (X7) at (6,4) {$X^{k-1}$};
\node (X8) at (-6,2) {$0$};
\node (X9) at (-4,2) {$Y_k^n$};
\node (X10) at (-2,2) {$Y_k^{n-1}$};
\node (X11) at (0,2) {$\ldots$};
\node (X12) at (2,2) {$Y_k^2$};
\node (X13) at (4,2) {$Y_k^1$};
\node (X14) at (6,2) {$X^k$};
\node (X15) at (6,0) {$\vdots$};
\node (X16) at (-6,-2) {$0$};
\node (X17) at (-4,-2) {$Y_n^n$};
\node (X18) at (-2,-2) {$Y_n^{n-1}$};
\node (X19) at (0,-2) {$\ldots$};
\node (X20) at (2,-2) {$Y_n^2$};
\node (X21) at (4,-2) {$Y_n^1$};
\node (X22) at (6,-2) {$X^n$};
\node (X23) at (-6,-4) {$0$};
\node (X24) at (-4,-4) {$0$};
\node (X25) at (-2,-4) {$0$};
\node (X26) at (0,-4) {$\ldots$};
\node (X27) at (2,-4) {$0$};
\node (X28) at (4,-4) {$X^{n+1}$};
\node (X29) at (6,-4) {$X^{n+1}$};
\draw [->,thick,dotted] (X1) -- (X2) node [midway,above] {};
\draw [->,thick,dotted] (X2) -- (X3) node [midway,above] {$g_{k-1^n}$};
\draw [->,thick,dotted] (X3) -- (X4) node [midway,above] {$g_{k-1}^{n-1}$};
\draw [->,thick,dotted] (X4) -- (X5) node [midway,above] {$g_{k_1}^3$};
\draw [->,thick,dotted] (X5) -- (X6) node [midway,above] {$g_{k-1}^2$};
\draw [->,thick,dotted] (X6) -- (X7) node [midway,above] {$g_{k-1}^1$};
\draw [->,thick] (X8) -- (X9) node [midway,above] {};
\draw [->,thick] (X9) -- (X10) node [midway,above] {$g_k^n$};
\draw [->,thick] (X10) -- (X11) node [midway,above] {$g_k^{n-1}$};
\draw [->,thick] (X11) -- (X12) node [midway,above] {$g_k^3$};
\draw [->,thick] (X12) -- (X13) node [midway,above] {$g_k^2$};
\draw [->,thick] (X13) -- (X14) node [midway,above] {$g_k^1$};
\draw [->,thick] (X16) -- (X17) node [midway,above] {};
\draw [->,thick] (X17) -- (X18) node [midway,above] {$g_n^n$};
\draw [->,thick] (X18) -- (X19) node [midway,above] {$g_n^{n-1}$};
\draw [->,thick] (X19) -- (X20) node [midway,above] {$g_n^3$};
\draw [->,thick] (X20) -- (X21) node [midway,above] {$g_n^2$};
\draw [->,thick] (X21) -- (X22) node [midway,above] {$g_n^1$};
\draw [->,thick] (X23) -- (X24) node [midway,above] {};
\draw [->,thick] (X24) -- (X25) node [midway,above] {};
\draw [->,thick] (X25) -- (X26) node [midway,above] {};
\draw [->,thick] (X26) -- (X27) node [midway,right] {};
\draw [->,thick] (X27) -- (X28) node [midway,right] {};
\draw [double,-,thick] (X28) -- (X29) node [midway,right] {};
\draw [->,thick,dotted] (X2) -- (X8) node [midway,right] {};
\draw [->,thick,dotted] (X3) -- (X9) node [midway,left] {$p_{k-1}^{n-1}$};
\draw [->,thick,dotted] (X6) -- (X12) node [midway,left] {$p_{k-1}^{n-2}$};
\draw [->,thick,dotted] (X7) -- (X13) node [midway,left] {$p_{k-1}^0$};
\draw [->,thick] (X17) -- (X23) node [midway,right] {};
\draw [->,thick] (X18) -- (X24) node [midway,right] {};
\draw [->,thick] (X21) -- (X27) node [midway,right] {};
\draw [->,thick] (X22) -- (X28) node [midway,left] {$f^n$};
\draw [->,thick] (X7) -- (X14) node [midway,right] {$f^{k-1}$};
\draw [->,thick] (X14) -- (X15) node [midway,right] {$f^k$};
\draw [->,thick] (X15) -- (X22) node [midway,right] {$f^{n-1}$};
\draw [->,thick] (X22) -- (X29) node [midway,right] {$f^n$};
\end{tikzpicture}
\end{center}
Since $f^kf^{k-1}=0$ and $g_k^1$ is a weak kernel of $f^k$, there exists a morphism $p_{k-1}^0:X^{k-1}\rightarrow Y_k^1$ such that $f^{k-1}=g_k^1p_{k-1}^0$. Note that the dotted morphisms in the diagram are obtained by taking $n$-pullback of the complex
\begin{equation}
0\rightarrow Y_k^n\overset{g_k^n}{\longrightarrow}Y_k^{n-1}\overset{g_k^{n-1}}{\longrightarrow}\cdots \overset{g_k^3}{\longrightarrow}Y_k^2\overset{g_k^2}{\longrightarrow}Y_k^1 \notag
\end{equation}
along $p_{k-1}^0$. It is easy to see that for every $1\leq k\leq n$ the induced morphism
\begin{equation}
[p_{k-1}^0, g_k^2]:X^{k-1}\oplus Y_k^2\longrightarrow Y_k^1 \notag
\end{equation}
in the above diagram is an epimorphism and hence the diagram is both an $n$-pullback diagram and an $n$-pushout diagram. Now all required properties are followed from basic properties of $n$-pullback and $n$-pushout.
Now assume that for every $k$
\begin{equation}
[p_k^0, g_{k+1}^2]:X^k\oplus Y_{k+1}^2\rightarrow Y_{k+1}^1 \notag
\end{equation}
is an epimorphism. We show that the complex $X$ is a right $n$-exact sequence. It is clear that $f^n$ is an epimorphism. Let $u:X^{k+1}\rightarrow M$ be a morphism such that $uf^k=0$. Since $[p_k^0, g_{k+1}^2]:X^k\oplus Y_{k+1}^2\rightarrow Y_{k+1}^1$ is an epimorphism, $ug_{k+1}^1=0$. By assumption
\begin{equation}
0\rightarrow (Y_{k+2}^1,M)\rightarrow (X^{k+1}\oplus Y_{k+2}^2,M)\rightarrow (Y_{k+1}^1\oplus Y_{k+2}^3,M) \notag
\end{equation}
is exact and so there is a morphism $v:Y_{k+2}^1\rightarrow M$ such that $u=vp_{k+1}^0$ and $vg_{k+2}^2=0$. Since $g_{k+2}^1$ is a weak cokernel of $g_{k+2}^2$ there is a morphism $w:X^{k+2}\rightarrow M$ such that $v=wg_{k+2}^1$. Now it is easy to see that $u=wf^{k+1}$ and hence $f^{k+1}$ is a weak cokernel of $f^k$. Then the complex $X$ is a right $n$-exact sequence.
\end{proof}
\end{proposition}
\begin{remark}
We remind that in the proof of Proposition 3.13 of \cite{J} we can construct the following commutative diagram inductively.
\begin{diagram}\label{dia1}
\centering
\begin{tikzpicture}
\node (X1) at (-3,4.5) {$X^0$};
\node (X2) at (-1,4.5) {$X^1$};
\node (X3) at (1,4.5) {$\ldots$};
\node (X4) at (3,4.5) {$X^{n-1}$};
\node (X5) at (5,4.5) {$X^n$};
\node (X6) at (7,4.5) {$X^{n+1}$};
\node (X7) at (-3,3) {$Y_0^1$};
\node (X8) at (-1,3) {$Y_1^1$};
\node (X9) at (3,3) {$Y_{n-1}^1$};
\node (X10) at (5,3) {$Y_n^1$};
\node (X11) at (7,3) {$Y_{n+1}^1=X^{n+1}$};
\node (X12) at (-3,1.5) {$Y_0^2$};
\node (X13) at (-1,1.5) {$Y_1^2$};
\node (X14) at (3,1.5) {$Y_{n-1}^2$};
\node (X15) at (5,1.5) {$Y_n^2$};
\node (X16) at (7,1.5) {$Y_{n+1}^2=0$};
\node (X17) at (-3,0) {$\vdots$};
\node (X18) at (-1,0) {$\vdots$};
\node (X19) at (3,0) {$\vdots$};
\node (X20) at (5,0) {$\vdots$};
\node (X21) at (7,0) {$\vdots$};
\node (X22) at (-3,-1.5) {$Y_0^{n-1}$};
\node (X23) at (-1,-1.5) {$Y_1^{n-1}$};
\node (X24) at (3,-1.5) {$Y_{n-1}^{n-1}$};
\node (X25) at (5,-1.5) {$Y_n^{n-1}$};
\node (X26) at (7,-1.5) {$Y_{n+1}^{n-1}=0$};
\node (X27) at (-3,-3) {$Y_0^n$};
\node (X28) at (-1,-3) {$Y_1^n$};
\node (X29) at (3,-3) {$Y_{n-1}^n$};
\node (X30) at (5,-3) {$Y_n^n$};
\node (X31) at (7,-3) {$Y_{n+1}^n=0$};
\node (X32) at (-3,-4.5) {$0$};
\node (X33) at (-1,-4.5) {$0$};
\node (X34) at (3,-4.5) {$0$};
\node (X35) at (5,-4.5) {$0$};
\node (X36) at (7,-4.5) {$0$};
\draw [->,thick] (X1) -- (X2) node [midway,above] {};
\draw [->,thick] (X2) -- (X3) node [midway,above] {};
\draw [->,thick] (X3) -- (X4) node [midway,above] {};
\draw [->,thick] (X4) -- (X5) node [midway,right] {};
\draw [->,thick] (X5) -- (X6) node [midway,right] {};
\draw [->,thick] (X32) -- (X27) node [midway,right] {};
\draw [->,thick] (X27) -- (X22) node [midway,right] {};
\draw [->,thick] (X22) -- (X17) node [midway,right] {};
\draw [->,thick] (X17) -- (X12) node [midway,right] {};
\draw [->,thick] (X12) -- (X7) node [midway,above] {};
\draw [->,thick] (X7) -- (X1) node [midway,above] {};
\draw [->,thick] (X33) -- (X28) node [midway,above] {};
\draw [->,thick] (X28) -- (X23) node [midway,right] {};
\draw [->,thick] (X23) -- (X18) node [midway,right] {};
\draw [->,thick] (X18) -- (X13) node [midway,right] {};
\draw [->,thick] (X13) -- (X8) node [midway,right] {};
\draw [->,thick] (X8) -- (X2) node [midway,right] {};
\draw [->,thick] (X34) -- (X29) node [midway,right] {};
\draw [->,thick] (X29) -- (X24) node [midway,above] {};
\draw [->,thick] (X24) -- (X19) node [midway,above] {};
\draw [->,thick] (X19) -- (X14) node [midway,above] {};
\draw [->,thick] (X14) -- (X9) node [midway,right] {};
\draw [->,thick] (X9) -- (X4) node [midway,right] {};
\draw [->,thick] (X35) -- (X30) node [midway,right] {};
\draw [->,thick] (X30) -- (X25) node [midway,right] {};
\draw [->,thick] (X25) -- (X20) node [midway,right] {};
\draw [->,thick] (X20) -- (X15) node [midway,right] {};
\draw [->,thick] (X15) -- (X10) node [midway,above] {};
\draw [->,thick] (X10) -- (X5) node [midway,above] {};
\draw [->,thick] (X36) -- (X31) node [midway,right] {};
\draw [->,thick] (X31) -- (X26) node [midway,right] {};
\draw [->,thick] (X26) -- (X21) node [midway,right] {};
\draw [->,thick] (X21) -- (X16) node [midway,right] {};
\draw [->,thick] (X16) -- (X11) node [midway,above] {};
\draw [->,thick] (X11) -- (X6) node [midway,right] {$id$};
\draw [->,thick] (X1) -- (X8) node [midway,above] {};
\draw [->,thick] (X7) -- (X13) node [midway,right] {};
\draw [->,thick] (X22) -- (X28) node [midway,right] {};
\draw [->,thick] (X27) -- (X33) node [midway,right] {};
\draw [->,thick] (X4) -- (X10) node [midway,right] {};
\draw [->,thick] (X9) -- (X15) node [midway,right] {};
\draw [->,thick] (X24) -- (X30) node [midway,above] {};
\draw [->,thick] (X29) -- (X35) node [midway,right] {};
\draw [->,thick] (X5) -- (X11) node [midway,right] {};
\draw [->,thick] (X10) -- (X16) node [midway,right] {};
\draw [->,thick] (X25) -- (X31) node [midway,right] {};
\draw [->,thick] (X30) -- (X36) node [midway,right] {};
\end{tikzpicture}
\end{diagram}

First we construct the right-hand column and then by taking
$n$-pullback we construct another columns. We need to construct this diagram for a complex \begin{equation}
 X^0 \overset{f^0}{\rightarrow} X^1 \overset{f^1}{\rightarrow} \cdots \overset{f^{n-1}}{\rightarrow} X^n \overset{f^n}{\rightarrow} X^{n+1} \notag
\end{equation} in the rest of the paper.
\end{remark}
\begin{lemma}\label{lemma1}
Consider the following commutative diagram in an abelian category.
\begin{diagram}\label{dia2}
\centering
\begin{tikzpicture}
\node (X1) at (-3,4.5) {$A^0$};
\node (X2) at (-1.5,4.5) {$A^1$};
\node (X3) at (0,4.5) {$\ldots$};
\node (X4) at (1.5,4.5) {$A^{n-1}$};
\node (X5) at (3,4.5) {$A^n$};
\node (X6) at (4.5,4.5) {$A^{n+1}$};
\node (Z) at (6,4.5) {$0$};
\node (X7) at (-3,3) {$B_0^1$};
\node (X8) at (-1.5,3) {$B_1^1$};
\node (X9) at (1.5,3) {$B_{n-1}^1$};
\node (X10) at (3,3) {$B_n^1$};
\node (X11) at (4.5,3) {$B_{n+1}^1$};
\node (X12) at (-3,1.5) {$B_0^2$};
\node (X13) at (-1.5,1.5) {$B_1^2$};
\node (X14) at (1.5,1.5) {$B_{n-1}^2$};
\node (X15) at (3,1.5) {$B_n^2$};
\node (X16) at (4.5,1.5) {$B_{n+1}^2$};
\node (X17) at (-3,0) {$\vdots$};
\node (X18) at (-1.5,0) {$\vdots$};
\node (X19) at (1.5,0) {$\vdots$};
\node (X20) at (3,0) {$\vdots$};
\node (X21) at (4.5,0) {$\vdots$};
\node (X22) at (-3,-1.5) {$B_0^{n-1}$};
\node (X23) at (-1.5,-1.5) {$B_1^{n-1}$};
\node (X24) at (1.5,-1.5) {$B_{n-1}^{n-1}$};
\node (X25) at (3,-1.5) {$B_n^{n-1}$};
\node (X26) at (4.5,-1.5) {$B_{n+1}^{n-1}$};
\node (X27) at (-3,-3) {$B_0^n$};
\node (X28) at (-1.5,-3) {$B_1^n$};
\node (X29) at (1.5,-3) {$B_{n-1}^n$};
\node (X30) at (3,-3) {$B_n^n$};
\node (X31) at (4.5,-3) {$B_{n+1}^n$};
\node (X32) at (-3,-4.5) {$0$};
\node (X33) at (-1.5,-4.5) {$0$};
\node (X34) at (1.5,-4.5) {$0$};
\node (X35) at (3,-4.5) {$0$};
\node (X36) at (4.5,-4.5) {$0$};
\draw [->,thick] (X1) -- (X2) node [midway,above] {};
\draw [->,thick] (X2) -- (X3) node [midway,above] {};
\draw [->,thick] (X3) -- (X4) node [midway,above] {};
\draw [->,thick] (X4) -- (X5) node [midway,right] {};
\draw [->,thick] (X5) -- (X6) node [midway,right] {};
\draw [->,thick] (X6) -- (Z) node [midway,right] {};
\draw [->,thick] (X32) -- (X27) node [midway,right] {};
\draw [->,thick] (X27) -- (X22) node [midway,right] {};
\draw [->,thick] (X22) -- (X17) node [midway,right] {};
\draw [->,thick] (X17) -- (X12) node [midway,right] {};
\draw [->,thick] (X12) -- (X7) node [midway,above] {};
\draw [->,thick] (X7) -- (X1) node [midway,above] {};
\draw [->,thick] (X33) -- (X28) node [midway,above] {};
\draw [->,thick] (X28) -- (X23) node [midway,right] {};
\draw [->,thick] (X23) -- (X18) node [midway,right] {};
\draw [->,thick] (X18) -- (X13) node [midway,right] {};
\draw [->,thick] (X13) -- (X8) node [midway,right] {};
\draw [->,thick] (X8) -- (X2) node [midway,right] {};
\draw [->,thick] (X34) -- (X29) node [midway,right] {};
\draw [->,thick] (X29) -- (X24) node [midway,above] {};
\draw [->,thick] (X24) -- (X19) node [midway,above] {};
\draw [->,thick] (X19) -- (X14) node [midway,above] {};
\draw [->,thick] (X14) -- (X9) node [midway,right] {};
\draw [->,thick] (X9) -- (X4) node [midway,right] {};
\draw [->,thick] (X35) -- (X30) node [midway,right] {};
\draw [->,thick] (X30) -- (X25) node [midway,right] {};
\draw [->,thick] (X25) -- (X20) node [midway,right] {};
\draw [->,thick] (X20) -- (X15) node [midway,right] {};
\draw [->,thick] (X15) -- (X10) node [midway,above] {};
\draw [->,thick] (X10) -- (X5) node [midway,above] {};
\draw [->,thick] (X36) -- (X31) node [midway,right] {};
\draw [->,thick] (X31) -- (X26) node [midway,right] {};
\draw [->,thick] (X26) -- (X21) node [midway,right] {};
\draw [->,thick] (X21) -- (X16) node [midway,right] {};
\draw [->,thick] (X16) -- (X11) node [midway,above] {};
\draw [->,thick] (X11) -- (X6) node [midway,above] {};
\draw [->,thick] (X1) -- (X8) node [midway,above] {};
\draw [->,thick] (X7) -- (X13) node [midway,right] {};
\draw [->,thick] (X22) -- (X28) node [midway,right] {};
\draw [->,thick] (X27) -- (X33) node [midway,right] {};
\draw [->,thick] (X4) -- (X10) node [midway,right] {};
\draw [->,thick] (X9) -- (X15) node [midway,right] {};
\draw [->,thick] (X24) -- (X30) node [midway,above] {};
\draw [->,thick] (X29) -- (X35) node [midway,right] {};
\draw [->,thick] (X5) -- (X11) node [midway,right] {};
\draw [->,thick] (X10) -- (X16) node [midway,right] {};
\draw [->,thick] (X25) -- (X31) node [midway,right] {};
\draw [->,thick] (X30) -- (X36) node [midway,right] {};
\end{tikzpicture}
\end{diagram}

Assume that the right-hand column is exact and for every $k$ the mapping cone
\begin{equation}
0\rightarrow B_k^n\rightarrow B_k^{n-1}\oplus 0\rightarrow \cdots \rightarrow A^k\oplus B_{k+1}^2\rightarrow B_{k+1}^1 \notag
\end{equation}
of the following morphism of complexes in the diagram is exact.
\begin{center}
\begin{tikzpicture}
\node (X1) at (-4,1) {$B_k$};
\node (X2) at (-2,1) {$B_k^n$};
\node (X3) at (0,1) {$B_k^{n-1}$};
\node (X4) at (2,1) {$\ldots$};
\node (X5) at (4,1) {$B_k^1$};
\node (X6) at (6,1) {$A^k$};
\node (X7) at (-6,-1) {$B_{k+1}$};
\node (X8) at (-4,-1) {$0$};
\node (X9) at (-2,-1) {$B_{k+1}^n$};
\node (X10) at (0,-1) {$\ldots$};
\node (X11) at (2,-1) {$B_{k+1}^2$};
\node (X12) at (4,-1) {$B_{k+1}^1$};
\draw [->,thick] (X1) -- (X7) node [midway,left] {};
\draw [->,thick] (X2) -- (X8) node [midway,left] {};
\draw [->,thick] (X3) -- (X9) node [midway,left] {};
\draw [->,thick] (X5) -- (X11) node [midway,left] {};
\draw [->,thick] (X6) -- (X12) node [midway,left] {};
\draw [->,thick] (X2) -- (X3) node [midway,above] {};
\draw [->,thick] (X3) -- (X4) node [midway,above] {};
\draw [->,thick] (X4) -- (X5) node [midway,above] {};
\draw [->,thick] (X5) -- (X6) node [midway,above] {};
\draw [->,thick] (X8) -- (X9) node [midway,above] {};
\draw [->,thick] (X9) -- (X10) node [midway,above] {};
\draw [->,thick] (X10) -- (X11) node [midway,above] {};
\draw [->,thick] (X11) -- (X12) node [midway,above] {};
\end{tikzpicture}
\end{center}
Then the top row is exact if and only if for every $k$ the morphism $A^k\oplus B_{k+1}^2\rightarrow B_{k+1}^1$ is an epimorphism.
\begin{proof}
Since the right-hand column and all mapping cones are exact, all other columns are also exact. Assume that for every $k$ the morphism $A^k\oplus B_{k+1}^2\rightarrow B_{k+1}^1$ is an epimorphism. Consider the following commutative diagram:
\begin{center}
\begin{tikzpicture}
\node (X1) at (-3,2) {$A^{k-1}$};
\node (X2) at (0,2) {$A^k$};
\node (X3) at (3,2) {$A^{k+1}$};
\node (X4) at (-3,0) {$B_{k-1}^1$};
\node (X5) at (0,0) {$B_k^1$};
\node (X6) at (3,0) {$B_{k+1}^1$};
\node (X7) at (-3,-2) {$B_{k-1}^2$};
\node (X8) at (0,-2) {$B_k^2$};
\node (X9) at (3,-2) {$B_{k+1}^2$};
\node (X10) at (3,-4) {$B_{k+1}^3$};
\draw [->,thick] (X1) -- (X2) node [midway,above] {$f^{k-1}$};
\draw [->,thick] (X2) -- (X3) node [midway,above] {$f^k$};
\draw [->,thick] (X7) -- (X4) node [midway,right] {$g_{k-1}^2$};
\draw [->,thick] (X4) -- (X1) node [midway,right] {$g_{k-1}^1$};
\draw [->,thick] (X8) -- (X5) node [midway,right] {$g_k^2$};
\draw [->,thick] (X5) -- (X2) node [midway,right] {$g_k^1$};
\draw [->,thick] (X10) -- (X9) node [midway,right] {$g_{k+1}^3$};
\draw [->,thick] (X9) -- (X6) node [midway,right] {$g_{k+1}^2$};
\draw [->,thick] (X6) -- (X3) node [midway,right] {$g_{k+1}^1$};
\draw [->,thick] (X1) -- (X5) node [midway,above] {$p^{k-1}$};
\draw [->,thick] (X4) -- (X8) node [midway,above] {$p_{k-1}^1$};
\draw [->,thick] (X2) -- (X6) node [midway,above] {$p^{k}$};
\draw [->,thick] (X5) -- (X9) node [midway,above] {$p_k^1$};
\draw [->,thick] (X8) -- (X10) node [midway,above] {$p_{k-1}^2$};
\end{tikzpicture}
\end{center}
We want to show that $\Ker(f^k)\subseteq \Imm(f^{k-1})$. Let $a^k\in \Ker(f^k)$, then $g_{k+1}^1p^k(a^k)=0$. Since all columns are exact, $p^k(a^k)=g_{k+1}^2(b_{k+1}^2)$ for some $b_{k+1}^2\in B_{k+1}^2$ and so
\begin{equation}
[p^k,g_{k+1}^2](a^k,-b_{k+1}^2)=0, \notag
\end{equation}
where $[p^k,g_{k+1}^2]:A^k\oplus B_{k+1}^2\rightarrow B_{k+1}^1$ is the induced morphism. Since the mapping cones are exact,  there exist $b_k^1\in B_k^1$ and $b_{k+1}^3\in B_{k+1}^3$ such that
\begin{equation}
(a^k,-b_{k+1}^2)=\begin{pmatrix} -g_k^1 & 0\\ p_k^1 & g_{k+1}^3 \end{pmatrix}
(b_k^1,b_{k+1}^3). \notag
\end{equation}
This mean that $a^k=g_k^1(b_k^1)$. Thus $b_k^1=p^{k-1}(a^{k-1})+g_k^2(b_k^2)$ for some $a^{k-1}\in A^{k-1}$ and $b_k^2\in B_k^2$. Therefore
\begin{equation}
a^k=g_k^1(b_k^1)=g_k^1p^{k-1}(a^{k-1})=f^{k-1}(a^{k-1}) \notag
\end{equation}
and so $a^k\in \Imm(f^{k-1})$. An easy calculation shows that $\Imm(f^{k-1})\subseteq \Ker(f^k)$ and the result follows. Now assume that the top row is exact and let $b_k^1\in B_k^1$. Since the Diagram 2.3 is commutative, $f^kg_k^1(b_k^1)=0$. By assumption there exists $a^{k-1}\in A^{k-1}$ such that $f^{k-1}(a^{k-1})=g_k^1(b_k^1)$. Then $g_k^1(p^{k-1}(a^{k-1})-b_k^1)=0$ and hence by exactness of columns $p^{k-1}(a^{k-1})-b_k^1=g_k^2(b_k^2)$ for some $b_k^2\in B_k^2$. This means that $b_k^1=[p^{k-1}, g_k^2](a^{k-1},-b_k^2)$. Therefore $[p^{k-1}, g_k^2]$ is an epimorphism.

\end{proof}
\end{lemma}


\section{the category of group valued functors}

In this section we first provide some preliminaries on the functor category $(\mathcal{M},\mathcal{G})$ and then we construct the subcategory $\mathcal{L}_2(\mathcal{M},\mathcal{G})$ of absolutely pure group valued functors of $(\mathcal{M},\mathcal{G})$ which is an abelian category with injective cogenerator.

\subsection{$n$-exact functors} In this subsection we recall the definitions of left $n$-exact, right $n$-exact and $n$-exact functors from an $n$-abelian category to an abelian category and show that a functor is $n$-exact if and only if it is both left and right $n$-exact.

\begin{definition}\label{def1}\cite[Section 4.1]{Lu}
Let $\mathcal{M}$ be an $n$-abelian category, $\mathcal{A}$ an abelian category and $F:\mathcal{M}\rightarrow \mathcal{A}$ a covariant functor.
\begin{itemize}
\item[(i)] $F$ is called left $n$-exact if for any left $n$-exact sequence $X^0 \overset{f^0}{\rightarrow} X^1 \overset{f^1}{\rightarrow} \cdots \overset{f^{n-1}}{\rightarrow} X^n \overset{f^n}{\rightarrow} X^{n+1} \notag$ in $\mathcal{M}$, $0 \rightarrow F(X^0) \rightarrow F(X^1) \rightarrow \cdots \rightarrow F(X^n)
    \rightarrow F(X^{n+1}) \notag$ is an exact sequence of $\mathcal{A}$.
\item[(ii)] $F$ is called right $n$-exact if for any right $n$-exact sequence $X^0 \overset{f^0}{\rightarrow} X^1 \overset{f^1}{\rightarrow} \cdots \overset{f^{n-1}}{\rightarrow} X^n \overset{f^n}{\rightarrow} X^{n+1} \notag$ in $\mathcal{M}$, $F(X^0) \rightarrow F(X^1) \rightarrow \cdots \rightarrow F(X^n)
    \rightarrow F(X^{n+1})\rightarrow 0 \notag$ is an exact sequence of $\mathcal{A}$.
\item[(iii)] $F$ is called $n$-exact if for any $n$-exact sequence $X^0 \overset{f^0}{\rightarrow} X^1 \overset{f^1}{\rightarrow} \cdots \overset{f^{n-1}}{\rightarrow} X^n \overset{f^n}{\rightarrow} X^{n+1} \notag$ in $\mathcal{M}$, $0 \rightarrow F(X^0) \rightarrow F(X^1) \rightarrow \cdots \rightarrow F(X^n)
    \rightarrow F(X^{n+1})\rightarrow 0 \notag$ is an exact sequence of $\mathcal{A}$.
\end{itemize}
\end{definition}

The contravariant left $n$-exact (resp., right $n$-exact, $n$-exact) functors are defined similarly.

\begin{proposition}
Let $\mathcal{M}$ be an $n$-abelian category and $\mathcal{A}$ an abelian category. A covariant functor $F:\mathcal{M}\rightarrow \mathcal{A}$ is an $n$-exact functor if and only if it is both left and right $n$-exact functor.
\begin{proof}
If $F$ is left $n$-exact and right $n$-exact, then it is obvious that $F$ is $n$-exact. Now assume that $F$ is $n$-exact. We show that it is right $n$-exact.
Let
\begin{equation}
X:X^0 \overset{d^0}{\rightarrow} X^1 \overset{d^1}{\rightarrow} \cdots \overset{d^{n-1}}{\rightarrow} X^n \overset{d^n}{\rightarrow} X^{n+1} \notag
\end{equation}
be a right $n$-exact sequence in $\mathcal{M}$. First we construct the Diagram 2.2 for $X$. In the Diagram 2.2 of $X$ the complex
\begin{equation}
 Y_{n+1}^n \overset{g_{n+1}^n}{\rightarrow} Y_{n+1}^{n-1} \overset{g_{n+1}^{n-1}}{\rightarrow} \cdots \overset{g_{n+1}^2}{\rightarrow} Y_{n+1}^1 \overset{g_{n+1}^1}{\rightarrow} X^{n+1} \notag
\end{equation}
is an $n$-exact sequence. Since $X$ is right $n$-exact, by Proposition \ref{propositionmain}(iii), we can choose $Y_j^i$'s in a way that Diagram 2.1 is both an $n$-pullback and an $n$-pushout. The mapping cones of $n$-pullback $n$-pushout diagrams are also $n$-exact sequences. Thus applying $F$ they are sent to exact sequences in $\mathcal{A}$.
After applying $F$ we have a commutative Diagram 2.3, where $F(X^i)=A^i$ and $F(Y^i_j)=B^i_j$. Then by Lemma \ref{lemma1}, $F(X^0) \rightarrow F(X^1) \rightarrow \cdots \rightarrow F(X^n)
    \rightarrow F(X^{n+1})\rightarrow 0 \notag$ is an exact sequence. Similarly we can see that $F$ is a left $n$-exact functor.
\end{proof}
\end{proposition}

\subsection{Reminder on basic properties}
Let $\mathcal{M}$ be an $n$-abelian category, we denote by $(\mathcal{M},\mathcal{G})$ the category of all additive functors from $\mathcal{M}$ to the category of all abelian groups. In this subsection we recall some basic properties of the category $(\mathcal{M},\mathcal{G})$. The reader can find proofs in \cite{Fr}.

A category is called complete if all small limits exist in it.
Dually, a category is called cocomplete if all small colimits
exist in it.  A cocomplete abelian category with a generator in
which direct limit of exact sequences is exact, is called a
Grothendieck category. The category $(\mathcal{M},\mathcal{G})$ is
a complete abelian category such that all limits compute pointwise.
Thus it is not hard to see that it is a Grothendieck category. If
for every $X \in \mathcal{M}$ we denote the functor $\mathcal{M}(X,-) \in
(\mathcal{M},\mathcal{G})$ by $H^X$, then the Yoneda lemma states
that $H^X$ is a projective object, and $\Sigma_{X\in \mathcal{M}}
H^X$ is a generator for $(\mathcal{M},\mathcal{G})$, where $\Sigma_{X\in \mathcal{M}}H^X$ is the direct sum of all representable functors $H^X$ in $(\mathcal{M},\mathcal{G})$.

We need the following results in the rest of the paper.

\begin{theorem}$($\cite[Theorem 6.25]{Fr}$)$
In a Grothendieck category with a generator, every object has an injective envelope.
\end{theorem}

\begin{proposition}\label{pro}$($\cite[Proposition 3.37]{Fr}$)$
Let $\mathcal{A}$ be a complete abelian category with a generator. Every object in $\mathcal{A}$ may be embedded in an injective object if and only if $\mathcal{A}$ has an injective cogenerator.
\end{proposition}

Now we give a higher-dimensional version of \cite[Proposition 7.11]{Fr}.

\begin{proposition}\label{propositionf0}
If an object $E\in (\mathcal{M},\mathcal{G})$ is injective, then it is a right $n$-exact functor.
\end{proposition}
\begin{proof}
Let $X^0 \rightarrow X^1 \rightarrow \cdots \rightarrow X^n \rightarrow X^{n+1}$ be a right $n$-exact sequence. Then we have the following exact sequence in $(\mathcal{M},\mathcal{G})$
\begin{equation}
0\rightarrow H^{X^{n+1}} \rightarrow H^{X^n}\rightarrow \cdots \rightarrow H^{X^1} \rightarrow H^{X^0}.  \notag
\end{equation}
Since $\Hom(-,E):(\mathcal{M},\mathcal{G})\rightarrow \mathcal{G}$ is an exact functor, applying this functor to the above sequence and using the Yoneda lemma the result follows.

\end{proof}

\subsection{The subcategory of mono functors}
A right $n$-exact functor is $n$-exact if and only if it carries monomorphisms to monomorphisms. A functor $F\in (\mathcal{M},\mathcal{G})$ is called a mono functor if it preserves monomorphisms. Thus by proposition \ref{propositionf0} an injective mono functor is an $n$-exact functor. We denote by $\mathbb{M}(\mathcal{M},\mathcal{G})$, the full subcategory of $(\mathcal{M},\mathcal{G})$ consist of all mono functors.

A monomorphism $A\rightarrow B$ is called an essential extension of $A$, if for every nonzero monomorphism $B'\rightarrow B$, the intersection of the images of $A\rightarrow B$ and $B'\rightarrow B$ are nonzero.

The following lemma is a higher-dimensional version of \cite[Lemma 7.12]{Fr}.

\begin{lemma}
Let $M\in \mathbb{M}(\mathcal{M},\mathcal{G})$, and $M\hookrightarrow E$ be an essential extension of $M$ in $(\mathcal{M},\mathcal{G})$. Then $E\in \mathbb{M}(\mathcal{M},\mathcal{G})$.
\begin{proof}
Assume that $E$ is not mono. Then there is a monomorphism $f:X^0\rightarrow X^1$ and $0\neq x\in E(X^0)$ such that $E(f)(x)=0$. We produce a subfunctor $F$ of $E$ generated by $x$. For every $Y\in \mathcal{M}$, define
\begin{equation}
F(Y):=\{y\in E(Y) \vert \;\text{there is a} \;h:X^0\rightarrow Y \;\text{such that} \;E(h)(x)=y \}. \notag
\end{equation}
Since $M\subseteq E$ is essential, there is an object $Y^0$ such that $F(Y^0)\cap M(Y^0)\neq0$. Let $0\neq y\in F(Y^0)\cap M(Y^0)$. By the construction of $F$ there is a morphism $h:X^0\rightarrow Y^0$ such that  $E(h)(x)=y$. Let
\begin{center}
\begin{tikzpicture}
\node (X2) at (-2,1) {$X^0$};
\node (X3) at (0,1) {$X^1$};
\node (X4) at (2,1) {$\ldots$};
\node (X5) at (4,1) {$X^{n-1}$};
\node (X6) at (6,1) {$X^n$};
\node (X8) at (-2,-1) {$Y^0$};
\node (X9) at (0,-1) {$Y^1$};
\node (X10) at (2,-1) {$\ldots$};
\node (X11) at (4,-1) {$Y^{n-1}$};
\node (X12) at (6,-1) {$Y^n$};
\draw [->,thick] (X2) -- (X8) node [midway,left] {$h$};
\draw [->,thick] (X3) -- (X9) node [midway,left] {$t$};
\draw [->,thick] (X5) -- (X11) node [midway,left] {};
\draw [->,thick] (X6) -- (X12) node [midway,left] {};
\draw [->,thick] (X2) -- (X3) node [midway,above] {$f$};
\draw [->,thick] (X3) -- (X4) node [midway,above] {};
\draw [->,thick] (X4) -- (X5) node [midway,above] {};
\draw [->,thick] (X5) -- (X6) node [midway,above] {};
\draw [->,thick] (X8) -- (X9) node [midway,above] {$g$};
\draw [->,thick] (X9) -- (X10) node [midway,above] {};
\draw [->,thick] (X10) -- (X11) node [midway,above] {};
\draw [->,thick] (X11) -- (X12) node [midway,above] {};
\end{tikzpicture}
\end{center}
be an $n$-pushout diagram where the top row is right $n$-exact. By Theorem \ref{theorem1}, $g$ is a monomorphism. Since $M$ is a mono functor, then $M(g)(y)\neq 0$ and hence $E(g)(y)\neq 0$. On the other hand $E(g)(y)=E(g)oE(h)(x)=E(t)oE(f)(x)=0$, which is a contradiction.
\end{proof}
\end{lemma}

By the above lemma, $\mathbb{M}(\mathcal{M},\mathcal{G})\subseteq (\mathcal{M},\mathcal{G})$ is a full subcategory closed under subobjects, products and essential extensions. Therefore all results of Section 7.2 of \cite{Fr} are valid. We summarize basic results of Section 7.2 of \cite{Fr} in the following proposition.

An object $M$ of $\mathbb{M}(\mathcal{M},\mathcal{G})$ is called torsion if for any object $N$ of $\mathbb{M}(\mathcal{M},\mathcal{G})$, $\Hom(M,N)=0$.

\begin{proposition}\label{propositionf1}
The inclusion functor $I:\mathbb{M}(\mathcal{M},\mathcal{G}) \hookrightarrow (\mathcal{M},\mathcal{G})$ has a left adjoint $\mathbb{M}:(\mathcal{M},\mathcal{G}) \rightarrow \mathbb{M}(\mathcal{M},\mathcal{G})$ such that for each $F\in (\mathcal{M},\mathcal{G})$ the morphism $F\rightarrow \mathbb{M}(F)$ is an epimorphism. The kernel of this morphism is the maximal torsion subobject of $F$.
\end{proposition}

$\mathbb{M}(\mathcal{M},\mathcal{G})$ is not in general an abelian category. There may be a monomorphism in $\mathbb{M}(\mathcal{M},\mathcal{G})$ which is not a kernel of a morphism in $\mathbb{M}(\mathcal{M},\mathcal{G})$. To fix this we introduce the subcategory of absolutely pure functors.

\begin{definition}$($\cite[Page 144]{Fr}$)$
A subfunctor $F^{\prime}\subseteq F$ in $\mathbb{M}(\mathcal{M},\mathcal{G})$ is said to be a pure subfunctor if the quotient functor $\dfrac{F}{F^{\prime}}\in \mathbb{M}(\mathcal{M},\mathcal{G})$. A mono functor is called absolutely pure if and only if whenever it appears as a subfunctor of a mono functor it is a pure subfunctor. We denote by $\mathcal{L}_2(\mathcal{M},\mathcal{G})\subseteq \mathbb{M}(\mathcal{M},\mathcal{G})$ the full subcategory of absolutely pure functors.
\end{definition}

In Theorem \ref{theoremc} we show that a mono functor $M\in (\mathcal{M},\mathcal{G})$ is absolutely pure if and only if whenever apply $M$ to a left $n$-exact sequence $X^0\rightarrow X^1 \rightarrow \ldots \rightarrow X^n\rightarrow X^{n+1}$ in $\mathcal{M}$, then $0\rightarrow M(X^0)\rightarrow M(X^1)\rightarrow M(X^2)$ is an exact sequence of abelian groups. For this nice property we use the symbol $\mathcal{L}_2(\mathcal{M},\mathcal{G})$ for the full subcategory of absolutely pure functors.

In the following proposition, which is special case of Theorems 7.28 and 7.29 of \cite{Fr}, we summarize the basic properties of the category $\mathcal{L}_2(\mathcal{M},\mathcal{G})$.

\begin{proposition}\label{propositionf2}
The inclusion functor $I:\mathcal{L}_2(\mathcal{M},\mathcal{G})\hookrightarrow \mathbb{M}(\mathcal{M},\mathcal{G})$ has a left adjoint $R:\mathbb{M}(\mathcal{M},\mathcal{G})\rightarrow \mathcal{L}_2(\mathcal{M},\mathcal{G})$ such that for each $M\in \mathbb{M}(\mathcal{M},\mathcal{G})$ the morphism $M\rightarrow R(M)$ is a monomorphism.
\end{proposition}

The following theorem is special case of Theorem 7.31 of \cite{Fr}.

\begin{theorem}\label{theoremab}
$\mathcal{L}_2(\mathcal{M},\mathcal{G})$ is an abelian category and every object of $\mathcal{L}_2(\mathcal{M},\mathcal{G})$ has an injective envelope.
\end{theorem}

In the following remark we recall that how kernels and cokernels in $\mathcal{L}_2(\mathcal{M},\mathcal{G})$ are constructed.

\begin{remark}\label{r3}
Let $L_1\rightarrow L_2$ be any morphism in $\mathcal{L}_2(\mathcal{M},\mathcal{G})$. If $K\rightarrow L_1$ is a kernel of $L_1\rightarrow L_2$ in $(\mathcal{M},\mathcal{G})$ it is also kernel in $\mathcal{L}_2(\mathcal{M},\mathcal{G})$. If $L_2\rightarrow F$ is a cokernel of $L_1\rightarrow L_2$ in $(\mathcal{M},\mathcal{G})$ then the cokernel in $\mathcal{L}_2(\mathcal{M},\mathcal{G})$ is the composition $L_2\rightarrow F\rightarrow \mathbb{M}(F)\rightarrow R(\mathbb{M}(F))$. Thus a morphism in $\mathcal{L}_2(\mathcal{M},\mathcal{G})$ is epimorphism if and only if its cokernel in $(\mathcal{M},\mathcal{G})$ is torsion.
\end{remark}


\section{representation of $n$-abelian categories}

In this section, we first characterize the category of absolutely pure functors. Then we give a representation of $\mathcal{M}$ in $\mathcal{L}_2(\mathcal{M},\mathcal{G})$.
\begin{lemma}$($\cite[Lemma 2.64]{Fr}$)$ \label{lemma5}
Consider the following commutative diagram in an abelian category with exact columns and exact middle row.
 \begin{center}
\begin{tikzpicture}
\node (X1) at (-1.5,2) {$0$};
\node (X2) at (0,2) {$0$};
\node (X3) at (1.5,2) {$0$};
\node (X4) at (-3,.5) {$0$};
\node (X5) at (-1.5,.5) {$B_{11}$};
\node (X6) at (0,.5) {$B_{12}$};
\node (X7) at (1.5,.5) {$B_{13}$};
\node (X8) at (-3,-1) {$0$};
\node (X9) at (-1.5,-1) {$B_{21}$};
\node (X10) at (0,-1) {$B_{22}$};
\node (X11) at (1.5,-1) {$B_{23}$};
\node (X12) at (-3,-2.5) {$0$};
\node (X13) at (-1.5,-2.5) {$B_{31}$};
\node (X14) at (0,-2.5) {$B_{32}$};
\node (X15) at (-1.5,-4) {$0$};
\draw [->,thick] (X1) -- (X5) node [midway,above] {};
\draw [->,thick] (X2) -- (X6) node [midway,left] {};
\draw [->,thick] (X3) -- (X7) node [midway,right] {};
\draw [->,thick] (X4) -- (X5) node [midway,above] {};
\draw [->,thick] (X5) -- (X6) node [midway,left] {};
\draw [->,thick] (X6) -- (X7) node [midway,right] {};
\draw [->,thick] (X5) -- (X9) node [midway,above] {};
\draw [->,thick] (X6) -- (X10) node [midway,left] {};
\draw [->,thick] (X7) -- (X11) node [midway,right] {};
\draw [->,thick] (X8) -- (X9) node [midway,above] {};
\draw [->,thick] (X9) -- (X10) node [midway,left] {};
\draw [->,thick] (X10) -- (X11) node [midway,right] {};
\draw [->,thick] (X9) -- (X13) node [midway,above] {};
\draw [->,thick] (X10) -- (X14) node [midway,left] {};
\draw [->,thick] (X12) -- (X13) node [midway,right] {};
\draw [->,thick] (X13) -- (X14) node [midway,above] {};
\draw [->,thick] (X13) -- (X15) node [midway,left] {};
\end{tikzpicture}
\end{center}
Then the top row is exact if and only if the bottom row is exact.
\end{lemma}

Now we give a higher-dimensional version of \cite[Theorem 7.27]{Fr}.

\begin{theorem} \label{theoremc}
A mono functor $M\in (\mathcal{M},\mathcal{G})$ is absolutely pure if and only if whenever $X^0\rightarrow X^1 \rightarrow \ldots \rightarrow X^n\rightarrow X^{n+1}$ is a left $n$-exact sequence in $\mathcal{M}$, then $0\rightarrow M(X^0)\rightarrow M(X^1)\rightarrow M(X^2)$ is an exact sequence of abelian groups.
\begin{proof}
Consider the exact sequence $0\rightarrow M\rightarrow E\rightarrow F\rightarrow 0$, where $E$ is an injective envelope of $M$. First assume that $M$ is absolutely pure and consider an arbitrary left $n$-exact sequence $X^0\rightarrow X^1\rightarrow \ldots \rightarrow X^n\rightarrow X^{n+1}$. Since injective functors are $n$-exact, the middle row of the following diagram is exact. Therefore the following commutative diagram satisfies the assumptions of Lemma \ref{lemma5}.
 \begin{center}
\begin{tikzpicture}
\node (X1) at (-2,2) {$0$};
\node (X2) at (0,2) {$0$};
\node (X3) at (2,2) {$0$};
\node (X4) at (-4,.5) {$0$};
\node (X5) at (-2,.5) {$M(X^0)$};
\node (X6) at (0,.5) {$M(X^1)$};
\node (X7) at (2,.5) {$M(X^2)$};
\node (X8) at (-4,-1) {$0$};
\node (X9) at (-2,-1) {$E(X^0)$};
\node (X10) at (0,-1) {$E(X^1)$};
\node (X11) at (2,-1) {$E(X^2)$};
\node (X12) at (-4,-2.5) {$0$};
\node (X13) at (-2,-2.5) {$F(X^0)$};
\node (X14) at (0,-2.5) {$F(X^1)$};
\node (X15) at (-2,-4) {$0$};
\draw [->,thick] (X1) -- (X5) node [midway,above] {};
\draw [->,thick] (X2) -- (X6) node [midway,left] {};
\draw [->,thick] (X3) -- (X7) node [midway,right] {};
\draw [->,thick] (X4) -- (X5) node [midway,above] {};
\draw [->,thick] (X5) -- (X6) node [midway,left] {};
\draw [->,thick] (X6) -- (X7) node [midway,right] {};
\draw [->,thick] (X5) -- (X9) node [midway,above] {};
\draw [->,thick] (X6) -- (X10) node [midway,left] {};
\draw [->,thick] (X7) -- (X11) node [midway,right] {};
\draw [->,thick] (X8) -- (X9) node [midway,above] {};
\draw [->,thick] (X9) -- (X10) node [midway,left] {};
\draw [->,thick] (X10) -- (X11) node [midway,right] {};
\draw [->,thick] (X9) -- (X13) node [midway,above] {};
\draw [->,thick] (X10) -- (X14) node [midway,left] {};
\draw [->,thick] (X12) -- (X13) node [midway,right] {};
\draw [->,thick] (X13) -- (X14) node [midway,above] {};
\draw [->,thick] (X13) -- (X15) node [midway,left] {};
\end{tikzpicture}
\end{center}
Since $M$ is absolutely pure, the bottom row is exact. Therefore by Lemma \ref{lemma5} the top row is exact and the result follows.
Now assume that $M$ has desired property. Thus by Lemma \ref{lemma5}, for any exact sequence $0\rightarrow M\rightarrow E\rightarrow F\rightarrow 0$, $F$ is a mono functor. Now let $0\rightarrow M\rightarrow N\rightarrow P\rightarrow 0$ be an exact sequence in $(\mathcal{M},\mathcal{G})$ such that $N$ is a mono functor. We must show that $P$ is also a mono functor. Let $M\rightarrow E$ be an injective envelope of $M$ and construct the following commutative diagram.
 \begin{center}
\begin{tikzpicture}
\node (X1) at (-4,1) {$0$};
\node (X2) at (-2,1) {$M$};
\node (X3) at (0,1) {$N$};
\node (X4) at (2,1) {$P$};
\node (X5) at (4,1) {$0$};
\node (X6) at (-4,-1) {$0$};
\node (X7) at (-2,-1) {$M$};
\node (X8) at (0,-1) {$E$};
\node (X9) at (2,-1) {$F$};
\node (X10) at (4,-1) {$0$};
\draw [->,thick] (X1) -- (X2) node [midway,above] {};
\draw [->,thick] (X2) -- (X3) node [midway,left] {};
\draw [->,thick] (X3) -- (X4) node [midway,right] {};
\draw [->,thick] (X4) -- (X5) node [midway,above] {};
\draw [->,thick] (X2) -- (X7) node [midway,left] {$id$};
\draw [->,thick] (X3) -- (X8) node [midway,right] {};
\draw [->,thick] (X4) -- (X9) node [midway,above] {};
\draw [->,thick] (X6) -- (X7) node [midway,left] {};
\draw [->,thick] (X7) -- (X8) node [midway,right] {};
\draw [->,thick] (X8) -- (X9) node [midway,above] {};
\draw [->,thick] (X9) -- (X10) node [midway,left] {};
\end{tikzpicture}
\end{center}
By the dual of \cite[Proposition 2.12]{Bu}, the right hand square is pullback and pushout diagram and hence the induced sequence $0\rightarrow N\rightarrow P\oplus E\rightarrow F\rightarrow 0$ is exact. Since $\mathbb{M}(\mathcal{M},\mathcal{G})$ is extension closed, $P$ is a mono functor and the result follows.
\end{proof}
\end{theorem}

The representation functor $H:\mathcal{M}\rightarrow (\mathcal{M},\mathcal{G})$ sends an object $X\in \mathcal{M}$ to $H^X$ that is a left $n$-exact functor. Thus by the above theorem $H$ factor through the subcategory $\mathcal{L}_2(\mathcal{M},\mathcal{G})$ and we have the following commutative diagram:
 \begin{center}
\begin{tikzpicture}
\node (X1) at (-3,1) {$\mathcal{M}$};
\node (X2) at (3,1) {$(\mathcal{M},\mathcal{G})$};
\node (X3) at (0,-2) {$\mathcal{L}_2(\mathcal{M},\mathcal{G})$};
\draw [->,thick] (X1) -- (X2) node [midway,above] {$H$};
\draw [->,thick] (X1) -- (X3) node [midway,left] {$\widetilde{H}$};
\draw [thick, right hook->] (X3) -- (X2) node [midway,right] {};
\end{tikzpicture}
\end{center}

\begin{theorem}\label{theoremmain3}
The functor $\widetilde{H}:\mathcal{M}\rightarrow \mathcal{L}_2(\mathcal{M},\mathcal{G})$ is an $n$-exact functor.
\begin{proof}
Let $X^0\rightarrow X^1\rightarrow \ldots \rightarrow X^n\rightarrow X^{n+1}$ be an $n$-exact sequence. We must show that $0\rightarrow H^{X^{n+1}}\rightarrow H^{X^n}\rightarrow \ldots \rightarrow H^{X^1}\rightarrow H^{X^0}\rightarrow 0$ is an exact sequence in $\mathcal{L}_2(\mathcal{M},\mathcal{G})$. By Proposition \ref{pro}, $\mathcal{L}_2(\mathcal{M},\mathcal{G})$ has an injective cogenerator. Let $E$ be an injective cogenerator of $\mathcal{L}_2(\mathcal{M},\mathcal{G})$.  The sequence $0\rightarrow H^{X^{n+1}}\rightarrow H^{X^n}\rightarrow \ldots \rightarrow H^{X^1}\rightarrow H^{X^0}\rightarrow 0$ is exact if and only if the induced sequence
\begin{equation}
0\rightarrow \Hom(H^{X^0},E)\rightarrow \Hom(H^{X^1},E)\rightarrow \cdots \rightarrow \Hom(H^{X^n},E)\rightarrow \Hom(H^{X^{n+1}},E)\rightarrow 0  \notag
\end{equation}
is exact. By the Yoneda lemma the above sequence is isomorphic to the sequence
\begin{equation}
0\rightarrow E(X^0)\rightarrow E(X^1)\rightarrow \cdots \rightarrow E(X^n)\rightarrow E(X^{n+1})\rightarrow 0  \notag
\end{equation}
which is exact, because $E$ is an $n$-exact functor.
\end{proof}
\end{theorem}

\begin{lemma}\label{lemmam}
Let $f:X\rightarrow Y$ be a morphism in $\mathcal{M}$.
\begin{itemize}
\item[(i)]
If $H^f:H^Y\rightarrow H^X$ is a monomorphism in $\mathcal{L}_2(\mathcal{M},\mathcal{G})$, then $f$ is an epimorphism.
\item[(ii)]
If $H^f:H^Y\rightarrow H^X$ is an epimorphism in $\mathcal{L}_2(\mathcal{M},\mathcal{G})$, then $f$ is a monomorphism.
\end{itemize}
\begin{proof}
\begin{itemize}
\item[(i)] Let $X^0\overset{d^0}{\rightarrow} X^1\overset{d^1}{\rightarrow} \cdots \overset{d^{n-2}}{\rightarrow} X^{n-1}\overset{d^{n-1}}{\rightarrow} X$ be an $n$-kernel of $f$. Then applying $\widetilde{H}$ we have that $0\rightarrow H^Y\rightarrow H^X\rightarrow H^{X^{n-1}}$ is an exact sequence in $\mathcal{L}_2(\mathcal{M},\mathcal{G})$. Since the kernel in $\mathcal{L}_2(\mathcal{M},\mathcal{G})$ coincides with the kernel in $(\mathcal{M},\mathcal{G})$ by Remark \ref{r3}, it is obvious that $f$ is a cokernel of $X^{n-1}\rightarrow X$ and so is epimorphism.
\item[(ii)]
By assumption $H^f:H^Y\rightarrow H^X$ is an epimorphism in $\mathcal{L}_2(\mathcal{M},\mathcal{G})$ and then, by Remark \ref{r3} its cokernel is torsion in $(\mathcal{M},\mathcal{G})$. Then the functor $\dfrac{\mathcal{M}(X,-)}{\Imm H^f}$ is a torsion functor. If $f$ is not a monomorphism, then there is an object $Z$ and a morphism $0\neq g:Z\rightarrow X$ such that $fg=0$. Then there exists a nonzero natural transformation $\dfrac{\mathcal{M}(X,-)}{\Imm H^f}\rightarrow \mathcal{M}(Z,-)$. Therefore $\dfrac{\mathcal{M}(X,-)}{\Imm H^f}$ is not torsion which is a contradiction.
\end{itemize}
\end{proof}
\end{lemma}

In the following theorem we show that the functor $\widetilde{H}:\mathcal{M}\rightarrow \mathcal{L}_2(\mathcal{M},\mathcal{G})$ reflects $n$-exactness.

\begin{theorem}\label{theoremmain2}
The functor $\widetilde{H}:\mathcal{M}\rightarrow \mathcal{L}_2(\mathcal{M},\mathcal{G})$ satisfies the following statements:
\begin{itemize}
\item[(i)]
Let $X:X^0\rightarrow X^1\rightarrow \cdots \rightarrow X^n\rightarrow X^{n+1}$ be a sequence of objects and morphisms in $\mathcal{M}$. If $0\rightarrow H^{X^{n+1}}\rightarrow H^{X^n}\rightarrow \cdots \rightarrow H^{X^1}\rightarrow H^{X^0}$ is an exact sequence in $\mathcal{L}_2(\mathcal{M},\mathcal{G})$ then $X$ is a right $n$-exact sequence.
\item[(ii)]
Let $X:X^0\rightarrow X^1\rightarrow \cdots \rightarrow X^n\rightarrow X^{n+1}$ be a sequence of objects and morphisms in $\mathcal{M}$. If $H^{X^{n+1}}\rightarrow H^{X^n}\rightarrow \cdots \rightarrow H^{X^1}\rightarrow H^{X^0}\rightarrow 0$ is an exact sequence in $\mathcal{L}_2(\mathcal{M},\mathcal{G})$ then $X$ is a left $n$-exact sequence.
\end{itemize}
\begin{proof}
We only prove the statement $(i)$ and the proof of the statement $(ii)$ is similar. Assume that $0\rightarrow H^{X^{n+1}}\rightarrow H^{X^n}\rightarrow \cdots \rightarrow H^{X^1}\rightarrow H^{X^0}$ is an exact sequence in $\mathcal{L}_2(\mathcal{M},\mathcal{G})$. Since $\widetilde{H}$ is faithful, $X$ is a complex in $\mathcal{M}$. Consider the Diagram 2.2 for $X$. Applying the functor $\widetilde{H}$ we have the Diagram 4.1 in the abelian category $\mathcal{L}_2(\mathcal{M},\mathcal{G})$. Since the notion of abelian category is self-dual, the dual of Lemma \ref{lemma1} is also correct. By assumption, the top row of the Diagram 4.1 is exact and hence by the dual of Lemma \ref{lemma1}, the morphism $d_C^k=\left[\begin{array}{c}
\widetilde{H}(p_k^0) \\
\widetilde{H}(g_{k+1}^2) \\
\end{array}\right]:\widetilde{H}(Y_{k+1}^1)\rightarrow \widetilde{H}(X^k)\oplus \widetilde{H}(Y_{k+1}^2)$ is a monomorphism. By Lemma \ref{lemmam}, $[p_k^0, g_{k+1}^2]:X^k\oplus Y_{k+1}^2\rightarrow Y_{k+1}^1$ is an epimorphism. Therefore by Proposition \ref{propositionmain}, the complex $X$ is a right $n$-exact sequence.
\begin{diagram}\label{dia1}
\centering
\begin{tikzpicture}
\node (X0) at  (-5,4.5) {$0$};
\node (X1) at (-3,4.5) {$\widetilde{H}(X^{n+1})$};
\node (X2) at (-1,4.5) {$\widetilde{H}(X^n)$};
\node (X3) at (1,4.5) {$\widetilde{H}(X^{n-1})$};
\node (X4) at (3,4.5) {$\ldots$};
\node (X5) at (5,4.5) {$\widetilde{H}(X^1)$};
\node (X6) at (7,4.5) {$\widetilde{H}(X^0)$};
\node (X7) at (-3,3) {$\widetilde{H}(X^{n+1})$};
\node (X8) at (-1,3) {$\widetilde{H}(Y_n^1)$};
\node (X9) at (1,3) {$\widetilde{H}(Y_{n-1}^1)$};
\node (X10) at (5,3) {$\widetilde{H}(Y_1^1)$};
\node (X11) at (7,3) {$\widetilde{H}(Y_0^1)$};
\node (X12) at (-3,1.5) {$0$};
\node (X13) at (-1,1.5) {$\widetilde{H}(Y_n^2)$};
\node (X14) at (1,1.5) {$\widetilde{H}(Y_{n-1}^2)$};
\node (X15) at (5,1.5) {$\widetilde{H}(Y_1^2)$};
\node (X16) at (7,1.5) {$\widetilde{H}(Y_0^2)$};
\node (X17) at (-3,0) {$\vdots$};
\node (X18) at (-1,0) {$\vdots$};
\node (X19) at (1,0) {$\vdots$};
\node (X20) at (5,0) {$\vdots$};
\node (X21) at (7,0) {$\vdots$};
\node (X22) at (-3,-1.5) {$0$};
\node (X23) at (-1,-1.5) {$\widetilde{H}(Y_n^{n-1})$};
\node (X24) at (1,-1.5) {$\widetilde{H}(Y_{n-1}^{n-1})$};
\node (X25) at (5,-1.5) {$\widetilde{H}(Y_1^{n-1})$};
\node (X26) at (7,-1.5) {$\widetilde{H}(Y_0^{n-1})$};
\node (X27) at (-3,-3) {$0$};
\node (X28) at (-1,-3) {$\widetilde{H}(Y_n^n)$};
\node (X29) at (1,-3) {$\widetilde{H}(Y_{n-1}^n)$};
\node (X30) at (5,-3) {$\widetilde{H}(Y_1^n)$};
\node (X31) at (7,-3) {$\widetilde{H}(Y_0^n)$};
\node (X32) at (-3,-4.5) {$0$};
\node (X33) at (-1,-4.5) {$0$};
\node (X34) at (1,-4.5) {$0$};
\node (X35) at (5,-4.5) {$0$};
\node (X36) at (7,-4.5) {$0$};
\draw [->,thick] (X0) -- (X1) node [midway,right] {};
\draw [->,thick] (X1) -- (X2) node [midway,above] {};
\draw [->,thick] (X2) -- (X3) node [midway,above] {};
\draw [->,thick] (X3) -- (X4) node [midway,above] {};
\draw [->,thick] (X4) -- (X5) node [midway,right] {};
\draw [->,thick] (X5) -- (X6) node [midway,right] {};
\draw [->,thick] (X27) -- (X32) node [midway,right] {};
\draw [->,thick] (X22) -- (X27) node [midway,right] {};
\draw [->,thick] (X17) -- (X22) node [midway,right] {};
\draw [->,thick] (X12) -- (X17) node [midway,right] {};
\draw [->,thick] (X7) -- (X12) node [midway,above] {};
\draw [->,thick] (X1) -- (X7) node [midway,above,left] {$id$};
\draw [->,thick] (X28) -- (X33) node [midway,above] {};
\draw [->,thick] (X23) -- (X28) node [midway,right] {};
\draw [->,thick] (X18) -- (X23) node [midway,right] {};
\draw [->,thick] (X13) -- (X18) node [midway,right] {};
\draw [->,thick] (X8) -- (X13) node [midway,right] {};
\draw [->,thick] (X2) -- (X8) node [midway,right] {};
\draw [->,thick] (X29) -- (X34) node [midway,right] {};
\draw [->,thick] (X24) -- (X29) node [midway,above] {};
\draw [->,thick] (X19) -- (X24) node [midway,above] {};
\draw [->,thick] (X14) -- (X19) node [midway,above] {};
\draw [->,thick] (X9) -- (X14) node [midway,right] {};
\draw [->,thick] (X3) -- (X9) node [midway,right] {};
\draw [->,thick] (X30) -- (X35) node [midway,right] {};
\draw [->,thick] (X25) -- (X30) node [midway,right] {};
\draw [->,thick] (X20) -- (X25) node [midway,right] {};
\draw [->,thick] (X15) -- (X20) node [midway,right] {};
\draw [->,thick] (X10) -- (X15) node [midway,above] {};
\draw [->,thick] (X5) -- (X10) node [midway,above] {};
\draw [->,thick] (X31) -- (X36) node [midway,right] {};
\draw [->,thick] (X26) -- (X31) node [midway,right] {};
\draw [->,thick] (X21) -- (X26) node [midway,right] {};
\draw [->,thick] (X16) -- (X21) node [midway,right] {};
\draw [->,thick] (X11) -- (X16) node [midway,above] {};
\draw [->,thick] (X6) -- (X11) node [midway,right]{} ;
\draw [->,thick] (X7) -- (X2) node [midway,above] {};
\draw [->,thick] (X12) -- (X8) node [midway,right] {};
\draw [->,thick] (X27) -- (X23) node [midway,right] {};
\draw [->,thick] (X32) -- (X28) node [midway,right] {};
\draw [->,thick] (X8) -- (X3) node [midway,right] {};
\draw [->,thick] (X13) -- (X9) node [midway,right] {};
\draw [->,thick] (X28) -- (X24) node [midway,above] {};
\draw [->,thick] (X33) -- (X29) node [midway,right] {};
\draw [->,thick] (X10) -- (X6) node [midway,right] {};
\draw [->,thick] (X15) -- (X11) node [midway,right] {};
\draw [->,thick] (X30) -- (X26) node [midway,right] {};
\draw [->,thick] (X35) -- (X31) node [midway,right] {};
\end{tikzpicture}
\end{diagram}
\end{proof}
\end{theorem}

We end this section with the following interesting question which is a goal of our next project.

\begin{question}
By the above notations let $\mathcal{L}(\mathcal{M},\mathcal{G})\subseteq \mathcal{L}_2(\mathcal{M},\mathcal{G})$ be the full subcategory of left $n$-exact functors.

Is $\mathcal{L}(\mathcal{M},\mathcal{G})$ an $n$-cluster tilting subcategory of $\mathcal{L}_2(\mathcal{M},\mathcal{G})$?

The positive answer to this question would show that any small $n$-abelian category is an exact full subcategory of an $n$-cluster tilting subcategory and, "who wants more? (S. MacLane)".
\end{question}


\section{applications}
Our main result in Section 4 implies that any statement about commutativity, ($n$-)exactness and existence of a morphism in a diagram in a small $n$-abelian category is true if and only if the corresponding statement is true in any abelian category. In this section, by using this fact, we prove some homological lemmas for $n$-abelian categories.

We need the following lemma in the proof of the next theorem.

\begin{lemma}\label{lemm4}
Let $\mathcal{M}$ be an $n$-abelian category, $\mathcal{A}$ an abelian category and $F:\mathcal{M}\rightarrow \mathcal{A}$ a functor.
\begin{itemize}
\item[(i)]
If $F$ is a left $n$-exact functor and $f:X\rightarrow Y$ is a weak kernel of $g:Y\rightarrow Z$, then $F(X)\rightarrow F(Y)\rightarrow F(Z)$ is exact.
\item[(ii)]
If $F$ is a right $n$-exact functor and $g:Y\rightarrow Z$ is a weak kernel of $f:X\rightarrow Y,$ then $F(X)\rightarrow F(Y)\rightarrow F(Z)$ is exact.
\end{itemize}
\begin{proof}
We only prove $(i)$, the proof of $(ii)$ is similar. If $n=1$ the proof is obvious. Assume that $n\geq 2$.  By the dual of \cite[Proposition 3.7]{J} we have the following left $n$-exact sequence
\begin{equation}
X^n\rightarrow \cdots \rightarrow X^3\rightarrow X^2 \rightarrow X\rightarrow Y\rightarrow Z. \notag
\end{equation}
Now by applying $F$, the statement follows.
\end{proof}
\end{lemma}

The following two theorems have been proved in \cite[Theorems 3.5 and 3.6]{Lu} for small projectively generated $n$-abelian categories. We prove these theorems for general $n$-abelian categories.

\begin{theorem}\textbf{(5-Lemma).}\label{Theorem5}
Let $\mathcal{M}$ be an $n$-abelian category, and
 \begin{center}
\begin{tikzpicture}
\node (X1) at (-4,2) {$X^0$};
\node (X2) at (-2,2) {$X^1$};
\node (X3) at (0,2) {$X^2$};
\node (X4) at (2,2) {$X^3$};
\node (X5) at (4,2) {$X^4$};
\node (X6) at (-4,0) {$Y^0$};
\node (X7) at (-2,0) {$Y^1$};
\node (X8) at (0,0) {$Y^2$};
\node (X9) at (2,0) {$Y^3$};
\node (X10) at (4,0) {$Y^4$};
\draw [->,thick] (X1) -- (X2) node [midway,above] {$d_X^0$};
\draw [->,thick] (X2) -- (X3) node [midway,above] {$d_X^1$};
\draw [->,thick] (X3) -- (X4) node [midway,above] {$d_X^2$};
\draw [->,thick] (X4) -- (X5) node [midway,above] {$d_X^3$};
\draw [->,thick] (X6) -- (X7) node [midway,above] {$d_Y^0$};
\draw [->,thick] (X7) -- (X8) node [midway,above] {$d_Y^1$};
\draw [->,thick] (X8) -- (X9) node [midway,above] {$d_Y^2$};
\draw [->,thick] (X9) -- (X10) node [midway,above] {$d_Y^3$};
\draw [->,thick] (X1) -- (X6) node [midway,right] {$f^0$};
\draw [->,thick] (X2) -- (X7) node [midway,right] {$f^1$};
\draw [->,thick] (X3) -- (X8) node [midway,right] {$f^2$};
\draw [->,thick] (X4) -- (X9) node [midway,right] {$f^3$};
\draw [->,thick] (X5) -- (X10) node [midway,right] {$f^4$};
\end{tikzpicture}
\end{center}
be a commuting diagram in $\mathcal{M}$. If $f^1$ and $f^3$ are isomorphisms, $f^0$ is an epimorphism, $f^4$ is a
monomorphism and one of the following conditions holds, then $f^2$ is also an isomorphism.
\begin{itemize}
\item[(i)]
$d_X^i$ and $d_Y^i$ are weak cokernels of $d_X^{i-1}$ and $d_Y^{i-1}$ respectively, for $i = 1, 2, 3, 4$.
\item[(ii)]
$d_X^i$ and $d_Y^i$ are weak kernels of $d_X^{i+1}$ and $d_Y^{i+1}$ respectively, for $i =0, 1, 2, 3$.
\end{itemize}
\begin{proof}
We can construct a small $n$-abelian category consisting of all objects in the diagram by adding $n$-kernel and $n$-cokernels. By Lemma \ref{lemm4}, both rows of the diagram are exact in the corresponding abelian category. Since the five lemma is valid in abelian categories, the result follows by Theorems \ref{theoremmain3} and \ref{theoremmain2}.
\end{proof}
\end{theorem}

\begin{theorem}$\textbf{((n+2)}\times \textbf{(n+2)-Lemma).}$
Let $\mathcal{M}$ be an $n$-abelian category and
 \begin{center}
\begin{tikzpicture}
\node (X1) at (-2,3) {$A^{1,1}$};
\node (X2) at (0,3) {$A^{1,2}$};
\node (X3) at (2,3) {$\ldots$};
\node (X4) at (4,3) {$A^{1,n+2}$};
\node (X5) at (-2,1) {$A^{2,1}$};
\node (X6) at (0,1) {$A^{2,2}$};
\node (X7) at (2,1) {$\ldots$};
\node (X8) at (4,1) {$A^{2,n+2}$};
\node (X9) at (-2,-1) {$\vdots$};
\node (X10) at (0,-1) {$\vdots$};
\node (X11) at (2,-1) {$\vdots$};
\node (X12) at (4,-1) {$\vdots$};
\node (X13) at (-2,-3) {$A^{n+2,1}$};
\node (X14) at (0,-3) {$A^{n+2,2}$};
\node (X15) at (2,-3) {$\ldots$};
\node (X16) at (4,-3) {$A^{n+2,n+2}$};
\draw [->,thick] (X1) -- (X2) node [midway,above] {};
\draw [->,thick] (X2) -- (X3) node [midway,above] {};
\draw [->,thick] (X3) -- (X4) node [midway,above] {};
\draw [->,thick] (X5) -- (X6) node [midway,above] {};
\draw [->,thick] (X6) -- (X7) node [midway,above] {};
\draw [->,thick] (X7) -- (X8) node [midway,above] {};
\draw [->,thick] (X13) -- (X14) node [midway,above] {};
\draw [->,thick] (X14) -- (X15) node [midway,above] {};
\draw [->,thick] (X15) -- (X16) node [midway,right] {};
\draw [->,thick] (X1) -- (X5) node [midway,right] {};
\draw [->,thick] (X5) -- (X9) node [midway,right] {};
\draw [->,thick] (X9) -- (X13) node [midway,right] {};
\draw [->,thick] (X2) -- (X6) node [midway,right] {};
\draw [->,thick] (X6) -- (X10) node [midway,right] {};
\draw [->,thick] (X10) -- (X14) node [midway,right] {};
\draw [->,thick] (X4) -- (X8) node [midway,right] {};
\draw [->,thick] (X8) -- (X12) node [midway,right] {};
\draw [->,thick] (X12) -- (X16) node [midway,right] {};
\end{tikzpicture}
\end{center}
be a commuting diagram in $\mathcal{M}$ such that all columns are $n$-exact. Then $n+1$ of
the $n+2$ rows are $n$-exact sequences implies the remaining row is also $n$-exact.
\begin{proof}
The proof is similar to the proof of Theorem \ref{Theorem5}.
\end{proof}
\end{theorem}

The following proposition was proved in \cite[Proposition 4.8]{J}. As an application of our main result in Section $4$, we give a much easier proof.

\begin{proposition}
Let $\mathcal{M}$ be an $n$-abelian category. Consider the following commutative diagram
\begin{diagram}\label{dia3}
\centering
\begin{tikzpicture}
\node (X1) at (-4,2) {$X^0$};
\node (X2) at (-2,2) {$X^1$};
\node (X3) at (0,2) {$\cdots$};
\node (X4) at (2,2) {$X^n$};
\node (X5) at (4,2) {$X^{n+1}$};
\node (X6) at (-4,0) {$Y^0$};
\node (X7) at (-2,0) {$Y^1$};
\node (X8) at (0,0) {$\cdots$};
\node (X9) at (2,0) {$Y^n$};
\draw [->,thick] (X1) -- (X2) node [midway,above] {$d_X^0$};
\draw [->,thick] (X2) -- (X3) node [midway,above] {$d_X^1$};
\draw [->,thick] (X3) -- (X4) node [midway,above] {$d_X^{n-1}$};
\draw [->,thick] (X4) -- (X5) node [midway,above] {$d_X^n$};
\draw [->,thick] (X6) -- (X7) node [midway,above] {$d_Y^0$};
\draw [->,thick] (X7) -- (X8) node [midway,above] {$d_Y^1$};
\draw [->,thick] (X8) -- (X9) node [midway,above] {$d_Y^{n-1}$};
\draw [->,thick] (X1) -- (X6) node [midway,right] {$f^0$};
\draw [->,thick] (X2) -- (X7) node [midway,right] {$f^1$};
\draw [->,thick] (X4) -- (X9) node [midway,right] {$f^n$};
\end{tikzpicture}
\end{diagram}
where the top row is an $n$-exact sequence and $d_Y^0$ is a monomorphism. Then the following statements are equivalent:
\begin{itemize}
\item[(i)]
The diagram is an $n$-pushout diagram.
\item[(ii)]
The mapping cone of the diagram is an $n$-exact sequence.
\item[(iii)]
The diagram is both an $n$-pushout and an $n$-pullback diagram.
\item[(iv)]
There exists a commutative diagram
 \begin{center}
\begin{tikzpicture}
\node (X1) at (-4,2) {$X^0$};
\node (X2) at (-2,2) {$X^1$};
\node (X3) at (0,2) {$\cdots$};
\node (X4) at (2,2) {$X^n$};
\node (X5) at (4,2) {$X^{n+1}$};
\node (X6) at (-4,0) {$Y^0$};
\node (X7) at (-2,0) {$Y^1$};
\node (X8) at (0,0) {$\cdots$};
\node (X9) at (2,0) {$Y^n$};
\node (X10) at (4,0) {$X^{n+1}$};
\draw [->,thick] (X1) -- (X2) node [midway,above] {$d_X^0$};
\draw [->,thick] (X2) -- (X3) node [midway,above] {$d_X^1$};
\draw [->,thick] (X3) -- (X4) node [midway,above] {$d_X^{n-1}$};
\draw [->,thick] (X4) -- (X5) node [midway,above] {$d_X^n$};
\draw [->,thick] (X6) -- (X7) node [midway,above] {$d_Y^0$};
\draw [->,thick] (X7) -- (X8) node [midway,above] {$d_Y^1$};
\draw [->,thick] (X8) -- (X9) node [midway,above] {$d_Y^{n-1}$};
\draw [->,thick] (X9) -- (X10) node [midway,above] {$d_Y^n$};
\draw [->,thick] (X1) -- (X6) node [midway,right] {$f^0$};
\draw [->,thick] (X2) -- (X7) node [midway,right] {$f^1$};
\draw [->,thick] (X4) -- (X9) node [midway,right] {$f^n$};
\draw [double,-,thick ] (X5) -- (X10) node [midway,right] {$id$};
\end{tikzpicture}
\end{center}
whose rows are $n$-exact sequences.
\end{itemize}
\begin{proof}
The equivalence of $(i)$, $(ii)$ and $(iii)$ is obvious.

$(i)\Rightarrow(iv)$. Since $\widetilde{H}:\mathcal{M}\rightarrow \mathcal{L}_2(\mathcal{M},\mathcal{G})$ is an embedding functor, for any $X\in \mathcal{M}$ we denote $\widetilde{H}(X)$ by $X$. Consider the Diagram 5.1. By $(ii)$
\begin{equation}
X^0\overset{d_C^0}{\longrightarrow} X^1\oplus Y^0 \overset{d_C^1}{\longrightarrow} \ldots \overset{d_C^{n-1}}{\longrightarrow} X^n\oplus Y^{n-1}\overset{d_C^n}{\longrightarrow} Y^n \notag
\end{equation}
is an $n$-exact sequence and applying $\widetilde{H}$ to it we get the following exact sequence in $\mathcal{L}_2(\mathcal{M},\mathcal{G})$.
\begin{equation}
0\rightarrow Y^n\overset{\widetilde{H}(d_C^n)}{\longrightarrow} X^n\oplus Y^{n-1} \overset{\widetilde{H}(d_C^{n-1})}{\longrightarrow} \ldots \overset{\widetilde{H}(d_C^1)}{\longrightarrow} X^1\oplus Y^{0}\overset{\widetilde{H}(d_C^0)}{\longrightarrow} Y^0\rightarrow0 \notag
\end{equation}
Since the top row is exact in $\mathcal{L}_2(\mathcal{M},\mathcal{G})$ and the mapping cone is exact, the bottom row is also exact. Therefore, by Theorem \ref{theoremmain2}, the bottom row is an $n$-exact sequence in $\mathcal{M}$.

$(iv)\Rightarrow(ii)$ is similar.
\end{proof}
\end{proposition}

\section*{acknowledgements}
The authors would like to thank the referee for a careful reading of this paper and
making helpful suggestions that improved the presentation of the paper. The research of the second
author was in part supported by a grant from IPM (No. 99170412).

\end{document}